%% file: groupesDyer.tex
\documentclass[12pt,a4paper]{article}

\usepackage{amsmath}
\usepackage{amsthm}
\usepackage{amssymb}
\usepackage{amsfonts}
\usepackage{tikz}
\usepackage{geometry}
\usepackage{enumitem}
\usepackage[numbers]{natbib}
\usepackage{xfrac}    
\usetikzlibrary{calc}
\usetikzlibrary{decorations.markings}
\usepackage{xcolor}
\usepackage{caption}
\usepackage{subcaption}

\definecolor{orange}{HTML}{F5793A}
\definecolor{violet}{HTML}{A95AA1}
\definecolor{dunkelblau}{HTML}{0F2080}
\definecolor{hellblau}{HTML}{85C0F9}

\tikzset{->-/.style={decoration={
			markings,
			mark=at position .6 with {\arrow{>}}},postaction={decorate}}}
\usepackage{hyperref}

\newtheorem{Satz}{Proposition}[section]

\newtheorem{Korollar}[Satz]{Corollary}
\newtheorem{Lemma}[Satz]{Lemma}
\newtheorem{Theorem}[Satz]{Theorem}

\newtheorem{Theorem*}[]{Theorem}
\newtheorem{Korollar*}[]{Corollary}

\theoremstyle{definition}
\newtheorem{Definition}[Satz]{Definition}

\theoremstyle{remark}

\newtheorem{Bemerkung}[Satz]{Remark}
\newtheorem{Beispiel}[Satz]{Example}
\newtheorem{Claim}[]{Claim}

\newtheorem{Fact}[Satz]{Fact}

\DeclareMathOperator{\N}{\mathbb{N}}
\DeclareMathOperator{\Id}{Id}
\DeclareMathOperator{\Z}{\mathbb{Z}}
\DeclareMathOperator{\R}{\mathbb{R}}

\DeclareMathOperator{\calS}{\cal S}
\DeclareMathOperator{\X}{{\cal X}}

\DeclareMathOperator{\D}{\mathfrak{D}}
\DeclareMathOperator{\Y}{{\cal Y}}
\DeclareMathOperator{\calP}{{\cal P}}
\DeclareMathOperator{\GX}{{\cal G(X)}}
\DeclareMathOperator{\GY}{{\cal G(Y)}}
\DeclareMathOperator{\GXY}{{\cal G(X)\times G(Y)}}
\DeclareMathOperator{\calZ}{{\cal Z}}
\DeclareMathOperator{\G}{\mathcal{G}}
\DeclareMathOperator{\calH}{{\cal H}}

\DeclareMathOperator{\Lk}{\mathrm{Lk}}
\DeclareMathOperator{\C}{{\cal C}}
\DeclareMathOperator{\Cox}{\mathrm{Cox}}
\DeclareMathOperator{\stern}{\mathrm{Stern}}

\DeclareMathOperator{\St}{\mathrm{St}}
\DeclareMathOperator{\F}{\cal F}

\DeclareMathOperator{\neutre}{\mathbf{e}}
\DeclareMathOperator{\rC}{\mathrm{Cc}}

\def\restriction#1#2{\mathchoice
	{\setbox1\hbox{${\displaystyle #1}_{\scriptstyle #2}$}
		\restrictionaux{#1}{#2}}
	{\setbox1\hbox{${\textstyle #1}_{\scriptstyle #2}$}
		\restrictionaux{#1}{#2}}
	{\setbox1\hbox{${\scriptstyle #1}_{\scriptscriptstyle #2}$}
		\restrictionaux{#1}{#2}}
	{\setbox1\hbox{${\scriptscriptstyle #1}_{\scriptscriptstyle #2}$}
		\restrictionaux{#1}{#2}}}
\def\restrictionaux#1#2{{#1\,\smash{\vrule height .8\ht1 depth .85\dp1}}_{\,#2}} 

\binoppenalty=\maxdimen
\relpenalty=\maxdimen

\newcommand{\Addresses}{{
		\bigskip
		\footnotesize
		
		Mireille Soergel, \textsc{ETH Z\"urich,
			Mathematics Department, 
			Rämistrasse 101,
			CH-8092 Zürich }\par\nopagebreak
		\textit{E-mail address}: \texttt{mireille.soergel@math.ethz.ch}

}}
\title{A generalization of the Davis-Moussong complex for Dyer groups}
\author{Mireille Soergel}

\begin{document}
 \maketitle
 \begin{abstract}
A common feature of Coxeter groups and right-angled Artin groups is their solution to the word problem. Matthew Dyer introduced a class of groups, which we call Dyer groups, sharing this feature. This class includes, but is not limited to, Coxeter groups, right-angled Artin groups, and graph products of cyclic groups. We introduce Dyer groups by giving their standard presentation and show that they are finite index subgroups of Coxeter groups. We then introduce a piecewise Euclidean cell complex $\Sigma$ which generalizes the Davis-Moussong complex and the Salvetti complex. The construction of $\Sigma$ uses simple categories without loops and complexes of groups. We conclude by proving that the cell complex $\Sigma$ is CAT(0).
 
 \end{abstract}

\input{Introduction2}

\input{FiniteIndex}

\input{Preliminaries}
\input{scwolv3}

\bibliographystyle{alpha}
\bibliography{MireilleBib}

\Addresses
\end{document}

%% file: Introduction2.tex
\section{Introduction}

There is extensive literature on Coxeter groups as well as on right-angled Artin groups and more generally graph products of cyclic groups. One common feature of these two families of groups is their solution to the word problem. It was given by Tits for Coxeter groups \cite{Tits} and by Green for graph products of cyclic groups \cite{Green}. The algorithm does not only give a solution to the word problem but also allows to detect whether a word is reduced or not. In his study of reflection subgroups of Coxeter groups, Dyer introduces a family of groups which contains both Coxeter groups and graph products of cyclic groups. A close study of \cite{Dyer} also implies that this class of groups, which we call Dyer groups, has the same solution to the word problem as Coxeter groups and graph products of cyclic groups. A complete and explicit proof is given in \cite{ParSoe} . 

Similarly to Coxeter groups and right-angled Artin groups, the presentation of a Dyer group can be encoded in a graph. Consider a simplicial graph $\Gamma$ with set of vertices $V = V(\Gamma)$ and set of edges $E = E(\Gamma)$, a vertex labeling $f : V\rightarrow \N_{\geq 2}\cup\{\infty\}$ and an edge labeling $m: E\rightarrow \N_{\geq 2}$. We say that the triple $(\Gamma, f, m)$ is a Dyer graph if for every edge $e = \{v,w\}$ with $f(v) \geq 3$ we have $m(e) = 2$. The associated Dyer group $D = D(\Gamma, f, m)$ is given by the following presentation \begin{multline*}
D = \langle x_v, v\in V \mid x_v^{f(v)} = \neutre \text{ if } f(v)\neq\infty, \\ [x_v,x_u]_{m(e)} = [x_u,x_v]_{m(e)} \text{ for all } e = \{u,v\} \in E \rangle,
\end{multline*} where $[a,b]_{k} = \underbrace{aba\dots}_{k}$ for any $a, b \in D$, $k\in\mathbb{N}$ and we denote the identity with $\neutre$.

It is natural to ask the following question. Consider a property P satisfied by both Coxeter groups and graph products of cyclic groups. Do Dyer groups also satisfy P?

In \cite{DavJan} Davis and Januszkiewicz show that right-angled Artin groups are finite index subgroups of right-angled Coxeter groups. For a right-angled Artin group $A$ they give right-angled Coxeter groups $W$ and $W'$ where $W'$ and $A$ are both finite index subgroups of $W$ and moreover the cubical complexes corresponding to $A$ and $W'$ are identical.
Inspired by this work we consider the following question: Are Dyer groups finite index subgroups of Coxeter groups? Out of a Dyer graph $(\Gamma, f, m)$ we build a labeled simplicial graph $\Lambda$ and prove the following statement.

\begin{Theorem}[Theorem \ref{Dyer}]
	$W(\Lambda) \cong D(\Gamma,f,m)\rtimes_{\xi} (\mathbb{Z}/2\mathbb{Z})^{k}$ for some determined $k\in\N$.
\end{Theorem}

The next corollary is a direct consequence.

\begin{Korollar}[Corollary \ref{Consequence}]
	Every Dyer group is a finite index subgroup of some Coxeter group.
\end{Korollar}

This has many interesting consequences, among others it implies that Dyer groups are CAT(0) \cite[Theorem 12.3.3]{Davis}, linear \cite[Corollary 2]{Bourba}, and biautomatic \cite{OsaPrz}. This is the starting point for a more precise study of their geometry. Coxeter groups are known to act geometrically by isometries on the Davis-Moussong complex, right-angled Artin groups are known to act geometrically by isometries on the Salvetti complex. Moreover graph products of cyclic groups are known to be CAT(0) by \cite[Theorem 8.20]{Genevo}. The aim is to construct an analog of the Davis-Moussong and Salvetti complexes for Dyer groups and by way of the construction give a unified description of them. The piecewise Euclidean cell complex $\Sigma$ associated to a Dyer group $D$ is constructed as follows. One considers a simple category without loops $\X$ and a complex of groups $\D(\X)$. The development $\C$ of $\D(\X)$ will then encode the necessary information to build $\Sigma$. In Section \ref{sec:spherical}, this is done for the case of spherical Dyer groups, where a Dyer group $D$ is spherical if it decomposes as a direct product $D_2\times D_{\infty}\times D_p$ with $D_2$ a finite Coxeter group, $D_{\infty} = \Z^n$ for some $n\in\N$ and $D_p$ a direct product of finite cyclic groups. In Section \ref{sec:general}, the construction of Section \ref{sec:spherical} is extended to any Dyer group. The complexes $\D(\X)$ and $\C$ are analogues to the poset of spherical subsets $\calS$ and the poset of spherical cosets $W\calS$ for Coxeter groups, which are recalled in Section \ref{Coxeter}. Finally Section \ref{sec:sigma} is devoted to the construction of the piecewise Euclidean cell complex $\Sigma$ and culminates with the proof of the following statement.

\begin{Theorem}[Theorem \ref{thm:cat}]
	The cell complex $\Sigma$ is CAT(0).
\end{Theorem}

As we do not assume the reader to be familiar with simple categories without loops (scwols), Section \ref{sec:scwols} recalls the definitions and statements needed for the construction of the scwol $\C$. In Sections \ref{Coxeter} and \ref{sec:Salvetti} we recall the constructions of the Davis-Moussong complex and of the Salvetti complex.

\paragraph{Acknowledgements} The author thanks her advisors Thomas Haettel and
Luis Paris for introducing her to Dyer's work. She is very grateful to them for many helpful discussions, comments and advice but also for their unwavering patience. The author is partially supported by the French project “AlMaRe” (ANR-19-CE40-0001-01) of the ANR.

%% file: FiniteIndex.tex
\section{Dyer groups}

We recall the definition of Dyer groups as given in the introduction. These groups were introduced by Dyer in \cite{Dyer}. It follows from Dyer's work that Dyer groups have the same solution to the word problem as Coxeter groups and right-angled Artin groups.

\begin{Definition} Let $\Gamma$ be a simplicial graph with set of vertices $V = V(\Gamma)$ and set of edges $E = E(\Gamma)$. Consider maps $f : V\rightarrow \N_{\geq 2}\cup\{\infty\}$ and $m: E\rightarrow \N_{\geq 2}$ such that for every edge $e = \{v,w\}$ with $f(v) \geq 3$ we have $m(e) = 2$. We call the triple $(\Gamma, f, m)$ a Dyer graph. 
	\end{Definition}

\begin{Definition}
	Let $(\Gamma, f, m)$ be a Dyer graph. The Dyer group $D = D(\Gamma, f, m)$ associated to the Dyer graph $(\Gamma, f, m)$ is given by the following presentation \begin{multline*}
	D = \langle x_v, v\in V \mid x_v^{f(v)} = \neutre \text{ if } f(v)\neq\infty, \\ [x_v,x_u]_{m(e)} = [x_u,x_v]_{m(e)} \text{ for all } e = \{u,v\} \in E \rangle,
	\end{multline*} where $[a,b]_{k} = \underbrace{aba\dots}_{k}$ for any $a, b \in D$, $k\in\mathbb{N}$ and we denote the identity with $\neutre$.
\end{Definition}

\begin{Beispiel}
 As mentioned in the introduction, Coxeter groups, right-angled Artin groups and graph products of cyclic groups are examples of Dyer groups.
\end{Beispiel}

\begin{Bemerkung}\label{SubDyer}
	For a subset $W\subset V$, we can consider $\Gamma_W$ the full subgraph of $\Gamma$ spanned by $W$ and the restrictions $f_W = \restriction{f}{W}$ and $m_W = \restriction{m}{E(\Gamma_W)}$. The triple $(\Gamma_W, f_W, m_W)$ is again a Dyer graph. We denote the associated Dyer group by $D_W$. From \cite{Dyer}, we know that  that the homomorphism $D_W \rightarrow D$ induced by the inclusion $W \hookrightarrow V$ is injective,
	hence $D_W$ can be regarded as a subgroup of $D$.
	
\end{Bemerkung}

\begin{Definition} Let $V_2 = \{v\in V\mid f(v) = 2\}$, $V_{\infty} = \{v\in V\mid f(v) = \infty\}$ and $V_p = V\setminus\{V_2\cup V_{\infty}\}$. For $i\in \{2, p, \infty\}$, let $\Gamma_i$ be the full subgraph spanned by $V_i$ and $D_i$ be the Dyer group associated to the triple $(\Gamma_i, f_{V_i}, m_{V_i})$. Note that $D_2$ is a Coxeter group, $D_{\infty}$ a right angled Artin group and $D_p$ a graph product of finite cyclic groups.
\end{Definition}

\begin{Beispiel}\label{nontrivial example}
	 Let $m,p\in\N_{\geq 2}$. Consider the Dyer graph $\Gamma_{m,p}$ given in Figure \ref{fig1}. The associated Dyer group is $$D_{m,p} = \langle a, b, c, d\mid b^2=c^2=d^p=\neutre, ab=ba,(bc)^m =\neutre, cd=dc\rangle.$$
\end{Beispiel}

\begin{figure}[h]
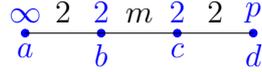

	\centering
	\include{graph1}
	\caption{Dyer graph $\Gamma_{m,p}$ for some $m, p\in\N_{\geq 2}$ }
	\label{fig1}
\end{figure}
We recall the definition of Coxeter groups.
\begin{Definition}
	Let $\Lambda$ be a simplicial graph with set of vertices $V = V(\Lambda)$ and set of edges $E = E(\Lambda)$. Let $m: E\rightarrow \N_{\geq 2}$ be an edge labeling of $\Lambda$. The Coxeter group $W = W(\Lambda)$ associated to the graph $\Lambda$ is given by the following presentation \begin{multline*}
	W = \langle x_v, v\in V \mid x_v^2 = \neutre \text{ for all } v\in V, \\ [x_v,x_u]_{m(e)} = [x_u,x_v]_{m(e)} \text{ for all } e = \{u,v\} \in E \rangle,
	\end{multline*} where $[a,b]_{k} = \underbrace{aba\dots}_{k}$ for any $a, b \in W$, $k\in\mathbb{N}$ and we denote the identity with $\neutre$. Note that for an edge $e=\{u,v\}\in E$ the relation $[x_v,x_u]_{m(e)} = [x_u,x_v]_{m(e)}$ is equivalent to the relation $(x_vx_u)^{m(e)}=\neutre$, since $x_u^2=x_v^2=\neutre$.
\end{Definition}

\paragraph{Dyer groups are finite index subgroups of Coxeter groups}

The aim is now to show that every Dyer group is a finite index subgroup of a Coxeter group. From a given Dyer graph $(\Gamma,f,m)$ we build a graph $\Lambda$ with edge labeling $m'$. We then show that $D(\Gamma,f,m)$ is a finite index subgroup of $W(\Lambda)$. See Example \ref{nontrivial Coxeter} for a simple case. We define the undirected labeled simplicial graph $\Lambda$. Its set of vertices is $V(\Lambda) = V \cup \{v'\mid v\in V_p\cup V_{\infty}\}$. Two vertices $u, v\in V\subset V(\Lambda)$ span an edge in $\Lambda$ if and only if they span an edge $e = \{u,v\}$ in $\Gamma$, we set the label of the edge $e = \{u,v\}\in E(\Lambda)$ to be $m'(e) = m(e)$. For all $u\in V_p\cup V_{\infty}$ and $v\in V(\Lambda)\setminus\{u, u'\}$, there is an edge $e = \{u', v\} \in E(\Lambda)$ labeled by $m'(e)=2$. Finally for all $u\in V_p$ there is an edge $e = \{u, u'\}\in E(\Lambda)$ labeled by $m'(e) = f(u)$. So $V\subset V(\Lambda)$ spans a copy of $\Gamma$ in $\Lambda$ and $V'_p\cup V'_{\infty}\subset V(\Lambda)$ spans a complete graph in $\Lambda$. Let $W = W(\Lambda)$ be the Coxeter group associated to the graph $\Lambda$. We give an action of $(\mathbb{Z}/2\mathbb{Z})^{|V_p\cup V_{\infty}|}$ on $D$. For $v\in V_p\cup V_{\infty}$ let $\xi_v: \{x_u, u\in V\}\rightarrow D$ with $\xi_v(x_u) = x_u$ for any $u\in V\setminus\{v\}$ and $\xi_v(x_v) = x_v^{-1}$. For all $v\in V_p\cup V_{\infty}$, the map $\xi_v$ extends to a homomorphism $\xi_v: D\rightarrow D$. Moreover for all $u, v\in V_p\cup V_{\infty}$, $\xi_v\circ \xi_u = \xi_u\circ \xi_v$ and $(\xi_v)^2 = \neutre$. Hence we have an action $\xi: (\mathbb{Z}/2\mathbb{Z})^{|V_p\cup V_{\infty}|}\times D \rightarrow D$.

\begin{Theorem}\label{Dyer}
	$W \cong D\rtimes_{\xi} (\mathbb{Z}/2\mathbb{Z})^{|V_p\cup V_{\infty}|}$.
\end{Theorem}

\begin{proof}
	Let us first recall the presentations of $W$, $D$ and $U = D\rtimes_{\xi} (\mathbb{Z}/2\mathbb{Z})^{|V_p\cup V_{\infty}|}$.
	\begin{multline*}W = \langle y_v, v\in V(\Lambda)\mid \forall v\in V(\Lambda), (y_v)^2=\neutre \\ \text{ and } \forall e = \{u,v\}\in E(\Lambda), (y_uy_v)^{m'(e)}=\neutre \rangle \end{multline*}
	 \begin{multline*}
	D = \langle x_v, v\in V \mid x_v^{f(v)} = \neutre \text{ if } f(v)\neq\infty, \\ [x_v,x_u]_{m(e)} = [x_u,x_v]_{m(e)} \text{ for all } e = \{u,v\} \in E \rangle
	\end{multline*}
	\begin{multline*}
	U = \langle \{x_u, u\in V\}\cup\{\xi_v, v\in V_p\cup V_{\infty}\}\mid x_u^{f(u)} = \neutre \text{ for all } u\in V \text{ with } f(u)\neq\infty, \\ [x_v,x_u]_{m(e)} = [x_u,x_v]_{m(e)} \text{ for all } e = \{u,v\} \in E \\
	\xi_v^2 = \neutre \text{ for all } v\in V_p\cup V_{\infty},\ \xi_v\xi_u=\xi_u\xi_v \text{ for all } u, v\in V_p\cup V_{\infty}\\
	\xi_ux_v=x_v\xi_u \text{ for all }u\in V_p\cup V_{\infty}, v\in V\setminus\{u\},\ \xi_ux_u\xi_u = x_u^{-1} \text{ for all } u\in V_p\cup V_{\infty}\rangle
	\end{multline*}
	
	We show that $U$ is isomorphic to $W$ by giving explicit homomorphisms \\ $\phi: W\rightarrow U$ and $\psi: U \rightarrow W$ satisfying $\phi\circ \psi = \Id_U$ and $\psi\circ\phi = \Id_W$.
	
	First consider the map $\phi: \{y_v, v\in V(\Lambda)\}\rightarrow U$ defined as follows: for $u\in V_2$, $\phi(y_u)=x_u$ and for $u\in V_p\cup V_{\infty}$, $\phi(y_u) = \xi_ux_u$ and $\phi(y_{u'}) = \xi_u$. We show that $\phi$ extends to a homomorphism $\phi: W\rightarrow U$.
	
	\begin{enumerate}
	\item For $u\in V_2$, $\phi(y_u)^2 =  x_u^2 =\neutre$. For $u\in V_p\cup V_{\infty}$, $\phi(y_{u'})^2=\xi_u^2=\neutre$ and $\phi(y_u)^2 = \xi_ux_u\xi_ux_u = x_u^{-1}x_u = \neutre$. So $\phi(y_u)^2 = \neutre$ for all $u\in V(\Lambda)$.
	\item Let $u,v\in V\subset V(\Lambda)$ with $ e=\{u,v\}\in E(\Lambda)$ so $e\in E$ and $m'(e)=m(e)$. If $u, v\in V_2$, 
	 we have $(\phi(y_u)\phi(y_v))^{m'(e)} = (x_ux_v)^{m(e)} = \neutre$ since $x_u^2=x_v^2=\neutre$ and hence $[x_u,x_v]_{m(e)}=[x_v,x_u]_{m(e)}$ is equivalent to $(x_ux_v)^{m(e)}=\neutre$. If $u\in V_2$ and $v\in V_p\cup V_{\infty}$, we have $m'(e)=2$ and so the relations in $U$ give $\phi(y_u)\phi(y_v) = x_u\xi_vx_v = \xi_vx_ux_v = \xi_vx_vx_u = \phi(y_v)\phi(y_u)$. If $u,v\in V_p\cup V_{\infty}$, $m'(e)=2$ and we have $\phi(y_u)\phi(y_v) = \xi_ux_u\xi_vx_v = \xi_vx_v\xi_ux_u = \phi(y_v)\phi(y_u)$.
	\item Let $u\in V_p\cup V_{\infty}$ and $v\in V\setminus\{u\}$. Then there is an edge $\{u',v\}\in E(\Lambda)$ with $m'(\{u',v\}) = 2$. If $v\in V_2$, $\phi(y_{u'})\phi(y_v) = \xi_ux_v = x_v\xi_u = \phi(y_v)\phi(y_{u'})$. If $v\in V_p\cup V_{\infty}\setminus\{u\}$, we have $\phi(y_{u'})\phi(y_v) = \xi_u \xi_vx_v = \xi_vx_v\xi_u = \phi(y_v)\phi(y_{u'})$.
	\item  Let $u\in V_p\cup V_{\infty}$ and $v\in (V_p\cup V_{\infty})\setminus\{u\}$, then there is an edge $\{u',v'\}\in E(\Lambda)$ with $m'(\{u',v'\}) = 2$ and we have $\phi(y_{u'})\phi(y_{v'}) = \xi_u \xi_v = \xi_v\xi_u = \phi(y_{v'})\phi(y_{u'})$.
	\item Finally for every $u\in V_p$ there is an edge $\{u',u\}\in E(\Lambda)$ with \linebreak $m'(\{u',u\}) = f(u)$ and $(\phi(y_{u'})\phi(y_{u}))^{f(u)} = (\xi_u\xi_ux_u)^{f(u)} = x_u^{f(u)} = \neutre$.
	
	\end{enumerate} So the map $\phi$ extends to a homomorphism $\phi: W\rightarrow U.$
	
	Now consider the map $\psi: \{x_u, u\in V\}\cup\{\xi_v, v\in V_p\cup V_{\infty}\} \rightarrow W$ defined as follows: for $u\in V_2$, $\psi(x_u) = y_u$ and for $u\in V_p\cup V_{\infty}$, $\psi(x_u) = y_{u'}y_u$ and $\psi(\xi_u) = y_{u'}$. We show that $\psi$ extends to a homomorphism from $U$ to $W$. 
	\begin{enumerate}\item For all $v\in V_2$, $f(v) = 2$ and so $\psi(x_v)^{f(v)} = y_v^2=\neutre$ and for all $v\in V_p$ there is an edge $e = \{v,v'\}\in E(\Lambda)$ with $m'(e) = f(v)$ so $\psi(x_v)^{f(v)} = (y_{v'}y_v)^{f(v)}=\neutre$. 
		\item For all $e = \{u,v\}\in E$ there is an edge $e = \{u,v\}\in E(\Lambda)$ with $m'(e)=m(e)$. If $u,v\in V_2$ we have $$[\psi(x_u),\psi(x_v)]_{m(e)} = [y_u,y_v]_{m(e)} = [y_v,y_u]_{m(e)} = [\psi(x_v),\psi(x_u)]_{m(e)}.$$ If $u\in V_2$ and $v\in V_p\cup V_{\infty}$, we have $m(e) = 2$ and $$\psi(x_u)\psi(x_v) = y_uy_{v'}y_v = y_{v'}y_vy_u = \psi(x_v)\psi(x_u).$$ If $u, v\in V_p\cup V_{\infty}$, $m(e) = 2$ and $$\psi(x_u)\psi(x_v) = y_{u'}y_uy_{v'}y_v = y_{v'}y_vy_{u'}y_u = \psi(x_v)\psi(x_u).$$
		\item For all $v\in V_p\cup V_{\infty}$ we have $\psi(\xi_v)^2 = y_{v'}^2 = \neutre$.
		\item For all $u, v\in V_p\cup V_{\infty}$, we have $e = \{u',v'\}\in E(\Lambda)$ with $m'(e)=2$, so $\psi(\xi_u)\psi(\xi_v) = y_{u'}y_{v'}= y_{v'}y_{u'} = \psi(\xi_v)\psi(\xi_u)$.
		\item For all $u\in V_p\cup V_{\infty}$ and $v\in V\setminus\{u\}$, we have $ \{u',v\}\in E(\Lambda)$ with $m'(\{u',v\})=2$ and $\{u',v'\}\in E(\Lambda)$ with $m'(\{u',v'\})=2$, so \linebreak 
	 $\psi(\xi_u)\psi(x_v) = \psi(x_v)\psi(\xi_u)$.
	 \item For all $u\in V_p\cup V_{\infty}$, $\psi(\xi_u)\psi(x_u)\psi(\xi_u) = y_{u'}y_{u'}y_uy_{u'} = y_uy_{u'} = \psi(x_u)^{-1}$.
	 \end{enumerate} So the map $\psi$ extends to a homomorphism $\psi: U\rightarrow W.$
	 
	 We now check that $\phi\circ \psi = \Id_U$ and $\psi\circ\phi = \Id_W$ by showing that these maps are the identity on the generators. For $v\in V_2$, we have $\phi(\psi(x_v)) = \phi(y_v) = x_v$ and $\psi(\phi(y_v)) = \psi(x_v) = y_v$. For $v\in V_p\cup V_{\infty}$, $\phi(\psi(x_v)) = \phi(y_{v'}y_v) = \xi_v\xi_vx_v = x_v$ and $\phi(\psi(\xi_v)) = \phi(y_{v'}) = \xi_v$. For $v\in V_p\cup V_{\infty}$, $\psi(\phi(y_v)) = \psi(\xi_vx_v) = y_{v'}y_{v'}y_v = y_v$ and $\psi(\phi(y_{v'})) = \psi(\xi_v) = y_{v'}$. 
\end{proof}

\begin{Korollar}\label{Consequence}
	Every Dyer group is a finite index subgroup of some Coxeter group. 
\end{Korollar}

\begin{Bemerkung}\label{CAT}
	As mentioned in the introduction, Corallary \ref{Consequence} has many interesting consequences. It implies that Dyer groups are CAT(0), linear, and biautomatic, that they satisfy the Baum-Connes conjecture, the Farrell-Jones conjecture, the Haagerup property and the strong Tits alternative. They also admit a proper and virtually special action on a CAT(0) cube complex.
\end{Bemerkung}

\begin{Beispiel}\label{nontrivial Coxeter}
	We apply the previous theorem to Example \ref{nontrivial example}. The corresponding graph $\Lambda$ is given in Figure \ref{fig2}. So by Theorem \ref{Dyer}, the Dyer group $D_{m,p}$ is an index $4$ subgroup of the Coxeter group
	
	 \begin{multline*}	
	W = \langle a, b, c, d, a', d'\mid a^2=b^2=c^2=d^2=a'^2=d'^2=\neutre, \\ (ab)^2= (bc)^m=(cd)^2=\neutre, (a'b)^2=(a'c)^2=(a'd)^2=(a'd')^2=\neutre, \\ (d'a)^2= (d'b)^2=(d'c)^2=\neutre, (d'd)^3=\neutre\rangle. \end{multline*}
\end{Beispiel}

\begin{figure}[!h]
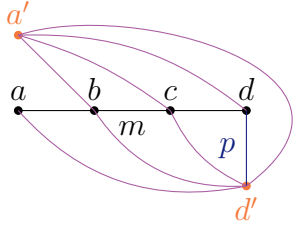

	\centering
	\include{graph1ext}
	\caption{The graph $\Lambda_{m,p}$ built out of the Dyer graph $\Gamma_{m,p}$ for some $m, p\in\N_{\geq 2}$. We color coded the vertices $V\subset V(\Lambda)$ and \textcolor{orange}{$\{v'\mid v\in V_p\cup V_{\infty}\}$}. For the edges: \textcolor{dunkelblau}{for edges of the form $e = \{u, u'\}$} and \textcolor{violet}{for edges of the form $e = \{u', v\}$, $v\in V(\Lambda)\setminus\{u, u'\}$ and $u'\in\{v'\mid v\in V_p\cup V_{\infty}\}$}. Every edge is labeled by $2$ if not specified otherwise.}
	\label{fig2}
\end{figure}

The Dyer group $D$ is not the only Dyer group which is a finite index subgroup of $W$. We describe such another Dyer group $D'= D(\Omega, g,n)$ by giving the Dyer graph $(\Omega, g, n)$. We define a second Dyer graph $(\Omega, g, n)$ starting from $\Gamma$. The vertices of $\Omega$ are $V(\Omega) = V\cup \{v' \mid v\in V_{\infty}\}$. The subsets $V_2\cup V_p\cup V_{\infty}$ and $V_2\cup V_p\cup \{v' \mid v\in V_{\infty}\}$ both span copies of $\Gamma$, with same labeling of edges, and for $u, v\in V_{\infty}$ the vertices $v, u'$ span an edge labeled by $2$ in $\Omega$ if and only if $v$ and $u$ span an edge in $\Gamma$. The labeling of the vertices is defined as follows $\restriction{g}{V_2\cup V_{\infty}\cup V_{\infty}'} = 2$ and $\restriction{g}{V_p} = \restriction{f}{V_p}$. Let $D'$ be the Dyer group associated to $\Omega$. Note that every generator $x_v$, $v\in V(\Omega)$ of $D'$ has finite order. We now give an action of $(\mathbb{Z}/2\mathbb{Z})^{|V_p\cup V_{\infty}|}$ on $D'$. For $v\in V_p$ let $\kappa_v: \{x_u, u\in V(\Omega)\}\rightarrow D'$ with $\kappa_v(x_u) = x_u$ for any $u\in V(\Omega)\setminus\{v \}$ and $\kappa_v(x_v) = x_v^{-1}$. For $v\in V_{\infty}$ let $\kappa_v: \{x_u, u\in V(\Omega)\}\rightarrow D'$ with $\kappa_v(x_u) = x_u$ for any $u\in V(\Omega)\setminus\{v, v'\}$ and $\kappa_v(x_v) = x_{v'}$ and $\kappa_v(x_{v'}) = x_{v}$. For all $v\in V_p\cup V_{\infty}$, the map $\kappa_v$ extends to a homomorphism $\kappa_v: D'\rightarrow D'$. Moreover for all $u, v\in V_p\cup V_{\infty}$, $\kappa_v\circ \kappa_u = \kappa_u\circ \kappa_v$ and $(\kappa_v)^2 = \neutre$. Hence we have an action $\kappa: (\mathbb{Z}/2\mathbb{Z})^{|V_p\cup V_{\infty}|}\times D' \rightarrow D'$.

\begin{Theorem}\label{Dyervar}
	$W \cong D'\rtimes_{\kappa} (\mathbb{Z}/2\mathbb{Z})^{|V_p\cup V_{\infty}|}$.
\end{Theorem}

\begin{proof}
	Let us first recall the presentations of $W$, $D'$ and $U = D'\rtimes_{\kappa} (\mathbb{Z}/2\mathbb{Z})^{|V_p\cup V_{\infty}|}$.
	\begin{multline*}W = \langle y_v, v\in V(\Lambda)\mid \forall v\in V(\Lambda), (y_v)^2=\neutre \\ \text{ and } \forall e = \{u,v\}\in E(\Lambda), (y_uy_v)^{m'(e)}=\neutre \rangle. \end{multline*}
	 \begin{multline*}
	D '= \langle x_v, v\in V(\Omega) \mid x_v^{g(v)} = \neutre \text{ for all } v\in V(\Omega), \\ [x_v,x_u]_{n(e)} = [x_u,x_v]_{n(e)} \text{ for all } e = \{u,v\} \in E(\Omega) \rangle
	\end{multline*}
	\begin{multline*}
	U = \langle \{x_u, u\in V(\Omega)\}\cup\{\kappa_v, v\in V_p\cup V_{\infty}\}\mid x_u^{g(u)} = \neutre \text{ for all } u\in V(\Omega), \\ [x_v,x_u]_{n(e)} = [x_u,x_v]_{n(e)} \text{ for all } e = \{u,v\} \in E(\Omega) \\
	\kappa_v^2 = \neutre \text{ for all } v\in V_p\cup V_{\infty},\ \kappa_v\kappa_u=\kappa_u\kappa_v \text{ for all } u, v\in V_p\cup V_{\infty}\\
	\kappa_ux_v=x_v\kappa_u \text{ for all }u\in V_p, v\in V(\Omega)\setminus\{u\},\\
	\kappa_ux_v=x_v\kappa_u \text{ for all }u\in V_{\infty}, v\in V(\Omega)\setminus\{u, u'\} \\
	 \kappa_ux_u\kappa_u = x_u^{-1} \text{ for all } u\in V_p,\ \kappa_ux_u\kappa_u = x_{u'} \text{ for all } u\in V_{\infty}\rangle
	\end{multline*}
	
	As in Theorem \ref{Dyer}, we can check that $U$ is isomorphic to $W$ by considering explicit homomorphisms $\phi: W\rightarrow U$ and $\psi: U \rightarrow W$ satisfying $\phi\circ \psi = \Id_U$ and $\psi\circ\phi = \Id_W$.
	
	The map $\phi: \{y_v, v\in V(\Lambda)\}\rightarrow U$ is given as follows: for $u\in V_2$, $\phi(y_u)=x_u$, for $u\in V_p$, $\phi(y_u) = \kappa_ux_u$ and $\phi(y_{u'}) = \kappa_u$ and for $u\in V_{\infty}$, $\phi(y_u) = x_u$ and $\phi(y_{u'}) = \kappa_u$. One can easily check, using methods which are similar to those used in the proof of Theorem \ref{Dyer}, that the map $\phi$ induces a homomorphism $\phi: W\rightarrow U$.
	
	The map $\psi: \{x_u, u\in V(\Omega)\}\cup\{\kappa_v, v\in V_p\cup V_{\infty}\} \rightarrow W$ is given as follows: for $u\in V_2$, $\psi(x_u) = y_u$ and for $u\in V_p$, $\psi(x_u) = y_{u'}y_u$ and $\psi(\kappa_u) = y_{u'}$, finally for $u\in V_{\infty}$, $\psi(x_u) = y_u$, $\psi(x_{u'}) = y_{u'}y_uy_{u'}$ and $\phi(\kappa_u) = y_{u'}$. Again one can easily check, using methods which are similar to those used in the proof of Theorem \ref{Dyer}, that the map $\psi$ induces a homomorphism $\psi: U\rightarrow W$, that $\phi\circ\psi = Id_U$ and that $\psi\circ\phi = Id_W$. 

\end{proof}

\begin{figure}[!h]
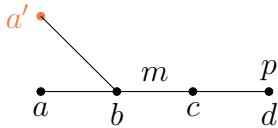

	\centering
	\include{graph1var}
	\caption{The graph $\Omega_{m,p}$ build out of the Dyer graph $\Gamma_{m,p}$ for some $m, p\in\N_{\geq 2}$. There are two types of vertices: $V\subset V(\Lambda_{m,p})$ and  \textcolor{orange}{$\{v'\mid v\in V_{\infty}\}$}. 
		Every vertex and every edge is labeled by $2$ if not specified otherwise.}
	\label{fig3}
\end{figure}

\begin{Beispiel}
	We apply the previous theorem to Example \ref{nontrivial example}. The corresponding graph $\Omega_{m,p}$ is given in Figure \ref{fig3}. The associated Dyer group is \begin{multline*}D'_{m,p} = \langle a, b, c, d, a'\mid a^2=a'^2= b^2=c^2=d^p=\neutre, \\ ab=ba, a'b=ba' ,(bc)^m =\neutre, cd=dc\rangle.\end{multline*} It is an index $4$ subgroup of the Coxeter group $W$ associated to the graph $\Lambda_{m,p}$ given in Figure \ref{fig2}.
\end{Beispiel}

\begin{Bemerkung}
	If the Dyer group $D$ is a right-angled Artin group, i.e. $V = V_{\infty}$, the constructions described here are those given in \cite{DavJan}. In particular if $D$ is a right-angled Artin group, the groups $W$ and $D'$ are right-angled Coxeter groups. So there is a decomposition of $W$ as a semi direct product of a right-angled Artin group and the right-angled Coxeter group $(\mathbb{Z}/2\mathbb{Z})^{|V|}$. 
\end{Bemerkung}


\begin{Bemerkung}
	There is a Coxeter group $W'$ associated to the Dyer group $D'$ such that $W' = D'\rtimes (\Z/2\Z)^{|V_p|}$. Question: do $W$ and $W'$ relate in any (meaningful) way? What can we say about their Davis-Moussong complexes? How do $D$ and $D'$ relate to each other? 
	What are all the Dyer subgroups of a given Coxeter group?
\end{Bemerkung}

%% file: graph1.tex
\begin{tikzpicture}

 \draw[fill, blue] (0,0) circle [radius=0.05];
 \draw[fill, blue] (1,0) circle [radius=0.05];
 \draw[fill,blue] (2,0) circle [radius=0.05];
 \draw[fill,blue] (3,0) circle [radius=0.05];
 \draw(0,0)node[above,blue]{$\infty$}node[below,blue]{$a$}--node[midway, above]{$2$}(1,0)node[above,blue]{$2$}node[below,blue]{$b$};
 \draw(1,0)--node[midway, above]{$m$}(2,0)node[above,blue]{$2$}node[below,blue]{$c$};
 \draw(2,0)--node[midway, above]{$2$}(3,0)node[above,blue]{$p$}node[below,blue]{$d$};
 
 \end{tikzpicture}

%% file: graph1ext.tex
\begin{tikzpicture}

 \draw[fill] (0,0) circle [radius=0.05]node[above]{$a$};
 \draw[fill] (1,0) circle [radius=0.05]node[above]{$b$};
 \draw[fill] (2,0) circle [radius=0.05]node[above]{$c$};
 \draw[fill] (3,0) circle [radius=0.05]node[above]{$d$};
 \draw[fill, color=orange] (0,1) circle [radius=0.05]node[above,color=orange]{$a'$};
 \draw[fill, color=orange] (3,-1) circle [radius=0.05]node[below,color=orange]{$d'$};
 \draw(0,0)--(1,0);
 \draw(1,0)--node[below]{$m$}(2,0);
 \draw(2,0)--(3,0);
 \draw(0,1)[color=violet]--(1,0);
 \draw(0,1)[color=violet]to[bend left=10](2,0);
 \draw(0,1)[color=violet]to[bend left=20](3,0);
 \draw(3,-1)[color=violet] to[bend left=30] (0,0);
  \draw(3,-1)[color=violet] to[bend left=30] (1,0);
 \draw(3,-1)[color=violet]to[bend left=20](2,0);
 \draw[color=dunkelblau](3,-1)--node[midway, left]{$p$}(3,0);
 \draw[color=violet](3,-1) .. controls (5,0.5) and (1.5,1.5) .. (0,1);
 
 \end{tikzpicture}

%% file: graph1var.tex
\begin{tikzpicture}

 \draw[fill] (0,0) circle [radius=0.05]node[below]{$a$};
 \draw[fill] (1,0) circle [radius=0.05]node[below]{$b$};
 \draw[fill] (2,0) circle [radius=0.05]node[below]{$c$};
 \draw[fill] (3,0) circle [radius=0.05]node[below]{$d$};
 \draw[fill, color=orange] (0,1) circle [radius=0.05]node[left, color=orange]{$a'$};
 \draw(0,0)--(1,0);
 \draw(1,0)--node[above]{$m$}(2,0);
 \draw(2,0)--(3,0)node[above]{$p$};
 \draw(0,1)--(1,0);
 
 \end{tikzpicture}

%% file: Preliminaries.tex
\section{Complexes of groups}
This section introduces the tools needed in Section \ref{Sigma}. We first recall necessary definitions and results about small categories without loops. We then recall the constructions of the Davis-Moussong complex and of the Salvetti complex.

\subsection{Introduction to scwols}\label{sec:scwols}

Small categories without loops (scwols) and complexes of groups were introduced by Haefliger in \cite{Haefli91}, \cite{Haefli92}. Based on \cite{BriHae}, we would like to recall some notions about scwols and complexes of groups, as we do not assume the reader to be familiar with these constructions. We hope to give all necessary definitions and results, details can be found in \cite{BriHae}. The reader familiar with scwols might only consider the specific examples developed in this section as they will be used in the construction of the cell complex $\Sigma$.

\input{IntroScwol}

\input{Coxeterv2}
\input{Salvetti}

%% file: IntroScwol.tex
A small category without loops (scwol) $\X$ consists of a set $V(\X)$, called the vertex set of $\X$ and a set $E(\X)$, called the set of edges of $\X$. Additionally two maps $i: E(\X)\rightarrow V(\X)$ and $t: E(\X)\rightarrow V(\X)$ are given. We call $i(\alpha)$ the initial vertex of $\alpha\in E(\X)$ and $t(\alpha)$ the terminal vertex of $\alpha\in E(\X)$. The set $E^{(2)}(\X)$ denotes the pairs $(\alpha,\beta)\in E(\X)\times E(\X)$ with $i(\alpha)=t(\beta)$. A third map $\circ: E^{(2)}(\X) \rightarrow E(\X)$ associates to each pair $(\alpha,\beta)$ an edge $\alpha\beta$, called their composition. These sets and maps need to satisfy the following conditions:
\begin{enumerate}
	\item for all $(\alpha,\beta)\in E^{(2)}$, we have $i(\alpha\beta) = i(\beta)$ and $t(\alpha\beta) = t(\alpha)$,
	\item for all $\alpha, \beta, \gamma\in E(\X)$, if $i(\alpha)=t(\beta)$ and $i(\beta)=t(\gamma)$, then $(\alpha\beta)\gamma = \alpha(\beta\gamma)$,
	\item for each $\alpha\in E(\X)$ we have $i(\alpha)\neq t(\alpha)$.
\end{enumerate}

A subscwol $\X'$ of $\X$ is given by subsets $V(\X')\subset V(\X)$ and $E(\X')\subset E(\X)$ such that if $\alpha\in E(\X')$, then $i(\alpha), t(\alpha)\in V(\X')$, and if $\alpha, \beta\in E(\X')$ with $i(\alpha)=t(\beta)$ then $\alpha\beta\in E(\X')$.

\begin{Bemerkung}\label{poset}
	To any poset $(\cal P, <)$ we can associate a scwol $\X$ where the set of vertices is $\cal P$ and the set of edges are pairs $(a,b)\in {\cal P}\times {\cal P}$ such that $b<a$, $t((a,b))=a$ and $i((a,b)) = b$. 
\end{Bemerkung}

\begin{Definition}[Product of scwols]\label{def product}
	Given two scwols $\X$ and $\Y$, their product $\X\times\Y$ is the scwol defined as follows: 
	$V(\X\times\Y) =V(\X)\times V(\Y)$ and $$E(\X\times\Y) = E(\X)\times V(\Y)\bigsqcup E(\X)\times E(\Y)\bigsqcup V(\X)\times E(\Y).$$ The maps $i, t : E(\X\times \Y) \rightarrow V(\X\times\Y)$ are defined by $i(\alpha,\beta) = (i(\alpha), i(\beta))$ and $t(\alpha,\beta) = (t(\alpha), t(\beta))$ (we consider $i(v) = t(v)=v$ for any $v\in V(\X)\bigsqcup V(\Y)$) and the composition $(\alpha,\alpha')(\beta,\beta') = (\alpha\beta, \alpha'\beta')$ whenever defined.
\end{Definition}

\begin{Bemerkung}
	Let $[n] = \{1,\dots, n\}$. One can inductively define the product of $n$ scwols $\X_1,\dots, \X_n$. Then the product $\X = \Pi_{i\in [n]} \X_i$ is the scwol with vertices $V(\X) =  \Pi_{i\in [n]} V(\X_i)$ and edges $ E(\X) = \bigsqcup_{S\subset[n], S\neq\emptyset} \left(\Pi_{i\in S^c}V(\X_i)\right)\times\left(\Pi_{i\in S}E(\X_i)\right).$ The maps $i, t : E(\X) \rightarrow V(\X)$ are defined by $i(\alpha)= i((\alpha_j)_{j\in [n]}) = (i(\alpha_j))_{j\in [n]}$ and $t(\alpha)=t((\alpha_j)_{j\in [n]}) = (t(\alpha_j))_{j\in [n]}$ (we consider $i(v) = t(v)=v$ for any $v\in \bigsqcup_{i\in [n]} V(\X_i)$) and the composition $\alpha\beta = (\alpha_j)_{j\in [n]}(\beta_j)_{j\in [n]} = (\alpha_j\beta_j)_{j\in [n]}$ whenever defined.
	
\end{Bemerkung}

\begin{Beispiel}\label{star}
	Consider a finite set $S$. Let $\calP(S)$ be the set of subsets of $S$. Consider the poset $(\calP(S), \subset)$ and its associated scwol $\Y_S$. Then $\Y_S = \Pi_{v\in S}\Y_{\{v\}}$. Moreover for any $v\in S$ the scwol $\Y_{\{v\}}$, also denoted $\Y_v$, has two vertices $\emptyset$ and $\{v\}$ and a single edge $e_v$ with $i(e_v)=\emptyset$ and $t(e_v) = \{v\}$.
\end{Beispiel}

\begin{Beispiel}\label{circle}
	Consider a finite set $S$. For $v\in S$ let $\calZ_{\{v\}} = \calZ_v$ be the scwol consisting of two vertices $\emptyset$ and $\{v\}$ and of two edges denoted $(\emptyset, \{v\}, \emptyset)$ and $(\emptyset, \{v\}, \{v\})$ with $i(\emptyset, \{v\}, \emptyset)=i(\emptyset, \{v\}, \{v\})=\emptyset$ and $t(\emptyset, \{v\}, \emptyset)=t(\emptyset, \{v\}, \{v\})=\{v\}$. Let $\calZ_S = \Pi_{v\in S}\calZ_{\{v\}}$. Note that $V(\calZ_S) = \calP(S)$ and the edges can be described as $E(\calZ_S) = \{(A, B, \lambda) \mid A\subsetneq B\subset S, \lambda\subset B\setminus A \}$, where $i(A,B,\lambda)=A$ and $t(A,B,\lambda)=B$. This example seems artificial at this point but will be quite useful later as the geometric realization of $\calZ_S$ is a torus $\mathbb{T}^{|S|}$ and its fundamental group is $\Z^{|S|}$. Indeed we are particularly interested in the case where $S$ is the vertex set of a complete Dyer graph $\Gamma$ for which all vertices are labeled by $\infty$.
	
\end{Beispiel}

A simple complex of groups ${\cal G(X)} = (G_v, \psi_{\alpha})$ over a scwol $\X$ is given by the following data: 
\begin{enumerate}
	\item for each $v\in V(\X)$, a group $G_v$ called the local group at $v$,
	\item for each $\alpha\in E(\X)$ an injective homomorphism $\psi_{\alpha} : G_{i(\alpha)}\rightarrow G_{t(\alpha)}$, with the following compatibility condition: $\psi_{\alpha\beta} = \psi_{\alpha}\psi_{\beta}$ whenever defined.
\end{enumerate}
A simple complex of groups ${\cal G(X)} = (G_v, \psi_{\alpha})$ over a scwol $\X$ is called inclusive if it additionally satisfies the following condition: for each $\alpha\in E(\X)$ we have $G_{i(\alpha)} <G_{t(\alpha)}$ and $\psi_{\alpha}(g) = g$ for all $g\in G_{i(\alpha)}$. We will only be considering 
simple inclusive complexes of groups. These are restrictions on the more general definition of complexes of groups which can be found in \cite{BriHae}. 
\begin{Definition}
	The product $\GXY$ of two simple complexes of groups $\GX$ and $\GY$ is the simple complex of groups over the scwol $\X\times\Y$ given by the following data: 
	\begin{enumerate}
		\item for each $v = (v_1,v_2)\in V(\X\times\Y)$, the local group $G_v = G_{v_1}\times G_{v_2}$ is the direct product of the corresponding local groups in $\GX$ and $\GY$,
		\item for each $\alpha = (\alpha_1, \alpha_2)\in E(\X\times\Y)$ the injective homomorphism is \linebreak $\psi_{\alpha} = \psi_{\alpha_1}\times\psi_{\alpha_2}$ (if one $\alpha_i$ is a vertex, we set $\psi_{\alpha_i}$ to be the identity on $G_{\alpha_i}$).
	\end{enumerate}

As $\GX$ and $\GY$ are simple complexes of groups so is $\GXY$.
	
\end{Definition}

Similarly to the definition of products of scwols this definition can be extended to finite products of simple inclusive complexes of groups.
The product $\Pi_{i\in [n]}\G(\X_i)$ of simple complexes of groups $\G(\X_i)$, $i\in[n]$, is the simple complex of groups over the scwol $\Pi_{i\in [n]}\X_i$ given by the following data: 
\begin{enumerate}
	\item for each $v = (v_i)_{i\in[n]}\in V(\Pi_{i\in [n]}\X_i)$, the local group $G_v = \Pi_{i\in [n]}G_{v_i}$ is the direct product of the corresponding local groups in $\G(\X_i)$,
	\item for each $\alpha = (\alpha_i)_{i\in[n]}\in E(\Pi_{i\in [n]}\X_i),$ the injective homomorphism is \linebreak $\psi_{\alpha} = \Pi_{i\in [n]}\psi_{\alpha_i}$ (if an $\alpha_i$ is a vertex, we set $\psi_{\alpha_i}$ to be the identity on $G_{\alpha_i}$).
\end{enumerate}

We will now give fundamental examples of complexes of groups and products of complexes of groups over the scwols introduced in Examples \ref{star} and \ref{circle}.

\begin{Beispiel}\label{cyclic star}
	We consider the scwols defined in Example \ref{star}. For every $v\in S$ we choose a positive integer $p_v$. Let $C_v$ be the finite cyclic group of order $p_v$. For $v\in S$, let $\D(\Y_v)$ be a simple complex of groups over $\Y_v$ by choosing $G_{\emptyset} = \langle \neutre \rangle$ the trivial group, $G_{\{v\}} = C_v$ and $\psi_{e_v}$ the trivial map. We define a simple complex of groups $\D(\Y_S)$ over $\Y_S$ as follows:
	\begin{enumerate}
		\item for $A\in V(\Y_S)$ we set $G_A = \Pi_{v\in A} C_v$, 
		\item for $e\in E(\Y_S)$ with $i(e) = A$ and $t(e) = B$ we have $A\subset B$ so $G_A<G_B$ and so there is a canonical inclusion $\psi_{e}: G_A\rightarrow G_B$. These inclusions satisfy the compatibility condition.
	\end{enumerate}
 We indeed have $\D(\Y_S) =\Pi_{v\in S}\D(\Y_{\{v\}})$. 
\end{Beispiel}

\begin{Beispiel}\label{Coxeter scwol}
 For a finite Coxeter system $(W, S)$ we can define $\mathfrak{W}(\Y_S)$ a simple complex of groups over $\Y_S$ as follows:
 \begin{enumerate}
 	\item for $A\in V(\Y_S)$ we choose the local group to be $W_A$, the subgroup of $W$ generated by $A$,
 	\item for $e\in E(\Y_S)$ with $i(e) = A$ and $t(e) = B$ we have $A\subset B$ so there is a canonical inclusion $\psi_{e}: W_A\rightarrow W_B$. These inclusions satisfy the compatibility condition.
 \end{enumerate}

In general in this case we have $\mathfrak{W}(\Y_S) \neq \Pi_{v\in S}\mathfrak{W}(\Y_{\{v\}})$ even though the scwols satisfy $\Y_S = \Pi_{v\in S}\Y_{\{v\}}$.
\end{Beispiel}

\begin{Beispiel}\label{cyclic circle}
	We consider the scwols defined in Example \ref{circle}. We can always define a trivial complex of groups over a scwol. The product of trivial complexes of groups will again be trivial. This leads to the following notation. For every $v\in S$ we define a simple complex of groups $\D(\calZ_v)$ over each scwol $\calZ_{v}$ by choosing $G_{\emptyset} = G_{\{v\}} = \langle \neutre \rangle$ the trivial group and $\psi_{(\emptyset, \{v\}, \emptyset)} = \psi_{(\emptyset, \{v\}, \{v\})}$ the trivial map. Similarly we define a simple complex of groups $\D(\calZ_S)$ by choosing $G_A = \langle \neutre \rangle$ the trivial group for every $A\in V(\calZ_S)$ and $\psi_{(A, B , \lambda)}$ the trivial map for every $(A, B, \lambda)\in E(\calZ_S)$. We have $\D(\calZ_S) =\Pi_{v\in S}\D(\calZ_{v})$.
\end{Beispiel}

Assume the scwol $\X$ is connected, i.e. there is only one equivalence class on $V(\X)$ for the equivalence relation generated by ($i(\alpha)\sim t(\alpha)$ for every edge $\alpha\in E(\X)$). 
One can define the fundamental group of a complex of groups ${\cal G(X)}$ over a scwol $\X$. For simplicity and as it suffices for the cases we consider, we give the following characterization.
\begin{Definition}Consider a simple complex of groups ${\cal G(X)}$ over a connected scwol $\X$. Assume each group $G_v$ is finitely presented with $G_v = \langle S_v \mid R_v\rangle$. Choose a maximal tree $T$ in the underlying graph. Let $E(\X)^{\pm} = \{\alpha^+, \alpha^-\mid \alpha\in E(\X)\}$. Then the fundamental group $\pi_1({\cal G(X)}, T)$ is generated by the set $$\left(\bigsqcup_{v\in V(\X)}S_v \right)\sqcup E(\X)^{\pm}$$ subject to the relations: 
	\begin{align*}
	&\text{ all the relations } R_v \text{ in the groups } G_v, \\
	&(\alpha^+)^{-1} = \alpha^- \text{ for all edges } \alpha\in E(\X),\\
	&\alpha^+\beta^+ = (\alpha\beta)^+ \text{ for } \alpha, \beta\in E(\X), \text{ whenever } \alpha\beta \in E(\X) \text{ is defined,} \\
	&\psi_{\alpha}(s) = \alpha^+s\alpha^-,\ \forall \alpha\in E(\X),\ \forall s\in S_{i(\alpha)},\\
	&\alpha^+ = \neutre,\ \forall \alpha\in T.
	\end{align*}

\end{Definition}
Different choices of $T$ will give isomorphic fundamental groups. So we can consider $\pi_1({\cal G(X)})=\pi_1({\cal G(X)}, T)$ for any choice of maximal tree $T$. Moreover the subgroup of $\pi_1({\cal G(X)}, T)$ generated by $\{\alpha^+, \alpha\in E(\X)\}$ corresponds to the fundamental group of the scwol $\X$.

\begin{Lemma} For two simple inclusive complexes of groups $\GX$ and $\GY$ we have 
	$$\pi_1(\GX)\times\pi_1(\GY)= \pi_1(\GXY).$$
\end{Lemma}

\begin{proof}We start by choosing maximal trees $T_{\X}$ in $\X$ and $T_{\Y}$ in $\Y$. In $\X\times\Y$ we fix a vertex $v_0 = (v_{\X,0}, v_{\Y,0})$ and consider the tree $T_{\X\times\Y}$ spanned by the  edges $\{(v, \alpha)\in E(\X\times\Y)\mid v\in V(\X), \alpha\in E(T_{\Y})\}\cup\{(\alpha, v_{\Y, 0})\in E(\X\times\Y)\mid \alpha\in E(T_{\X})\}$. Since $T_{\X}$ and $T_{\Y}$ are maximal trees, so is $T_{\X\times\Y}$. To prove the statement we give explicit homomorphisms $\phi: \pi_1({\cal G(X)}, T_{\X})\times \pi_1({\cal G(Y)}, T_{\Y})\rightarrow \pi_1({\cal G(X)\times G(Y)}, T_{\X\times\Y})$, and $\xi: \pi_1({\cal G(X)\times G(Y)}, T_{\X\times\Y})\rightarrow \pi_1({\cal G(X)}, T_{\X})\times \pi_1({\cal G(Y)}, T_{\Y})$ and show that their composition is the identity. Recall that the generating set of $\pi_1({\cal G(X)}, T_{\X})$ is $S_{\X} = \left(\bigsqcup_{v\in V(\X)}S_v \right)\sqcup E(\X)^{\pm}$ and that the generating set of $\pi_1({\cal G(Y)}, T_{\Y})$ is $S_{\Y} = \left(\bigsqcup_{v\in V(\Y)}S_v\right) \sqcup E(\Y)^{\pm}$. So the generating set of $\pi_1(\GX)\times\pi_1(\GY)$ is $S_{\X}\sqcup S_{\Y}$. This generating set is subject to the relations in $\pi_1(\GX)$ and $\pi_1(\GY)$ and to the commutation relations $\{st = ts \mid s\in S_{\X}, t\in S_{\Y}\}$. The generating set of $\pi_1({\cal G(X)\times G(Y)}, T_{\X\times\Y})$ is $S_{\X\times\Y} = \left(\bigsqcup_{v\in V(\X\times\Y)}S_v\right) \sqcup E(\X\times\Y)^{\pm}$. For a vertex $v = (v_{\X}, v_{\Y})\in V(\X\times\Y)$ we have $G_v=G_{v_{\X}}\times G_{v_{\Y}}$ so $S_v = S_{v_{\X}}\sqcup S_{v_{\Y}}$ and the relations in the groups $G_v$ are $R_v = R_{v_{\X}}\sqcup R_{v_{\Y}}\sqcup \{xy = yx\mid x\in S_{v_{\X}}, y\in S_{v_{\Y}}\}$. 
There is a lot of redundancy in the generators of the group $\pi_1({\cal G(X)\times G(Y)}, T_{\X\times\Y})$. In particular:
	\begin{enumerate}
		\item For all $\alpha\in E(\X)$, $v, w\in V(\Y)$ we have $(\alpha, v)^+ = (\alpha, w)^+$. First assume that $v$ and $w$ are adjacent in $T_{\Y}$. So there is an edge $\gamma\in T_{\Y}$ with $i(\gamma)=v$ and $t(\gamma)=w$. By Definition \ref{def product} we have $(t(\alpha), \gamma)(\alpha, v) = (\alpha, \gamma) = (\alpha, w)(i(\alpha), \gamma)$ in $\X\times\Y$. As $(t(\alpha), \gamma), (i(\alpha), \gamma)\in T_{\X\times\Y}$, we get $(t(\alpha), \gamma)^+=(i(\alpha), \gamma)^+=\neutre$, which implies $(\alpha, v)^+ = (\alpha, \gamma)^+ =(\alpha, w)^+$. If there is no such edge in $T_{\Y}$, there is sequence of vertices $v_1,\dots, v_k$ with $v_i = v$ and $v_k = w$ and such that for $1\leq i\leq k-1$ the vertices $v_i$, $v_{i+1}$ are adjacent in $T_{\Y}$. So for $1\leq i\leq k-1$ we have $(\alpha, v_i)^+ = (\alpha, v_{i+1})^+$ and hence $(\alpha, v)^+ = (\alpha, w)^+$. 
		\item The previous statement implies that for every $\gamma\in T_{\X}$ and every $v\in V(\Y)$ we have $(\gamma, v)^+ = \neutre$. So we can do a similar construction as for the previous statement to show that for all $\alpha\in E(\Y)$, $v, w\in V(\X)$ we have $(v,\alpha)^+ = (w,\alpha)^+$.
		\item For $\alpha\in E(\X)$, $\beta\in E(\Y)$, $v\in V(\X), w\in V(\Y)$ we have on one hand $(\alpha,\beta)^+ = (\alpha, t(\beta))^+(i(\alpha), \beta)^+ = (\alpha, w)^+(v,\beta)^+$ and on the other hand we have $(\alpha, \beta)^+ = (t(\alpha),\beta)^+(\alpha, i(\beta))^+ = (v, \beta)^+(\alpha, w)^+$. So we get \linebreak $(\alpha, \beta)^+= (\alpha, w)^+(v,\beta)^+=(v, \beta)^+(\alpha, w)^+$.
		\item Let $v\in V(\X)$, $\alpha\in T_{\Y}$ and $s\in S_v$. For $w\in V(\Y)$ write $s_{(v,w)}$ for $s\in S_{(v,w)} = S_v\sqcup S_w\subset S_{\X\times\Y}$. Then $s_{(v,i(\alpha))} = s_{(v,t(\alpha))}$ in\linebreak $ \pi_1({\cal G(X)\times G(Y)}, T_{\X\times\Y})$. Indeed $(v, \alpha)^+ = (v, \alpha)^-=\neutre$ and so \linebreak $\psi_{(v,\alpha)}(s_{(v,i(\alpha))}) = s_{(v,i(\alpha))}$ hence $s_{(v,t(\alpha))} = s_{(v,i(\alpha))}$. Moreover this implies that for all $w, w'\in V(\Y)$, $s_{(v,w)} = s_{(v,w')}$.
		
		\item Similarly let $v, v'\in V(\X)$, $w\in V(\Y)$, $s\in S_w$. We can write $s_{(v,w)}$ for $s\in S_{(v,w)} = S_v\sqcup S_w\subset S_{\X\times\Y}$. So we have $s_{(v,w)} =s_{(v',w)}$.
		\item Let $v\in V(\X), w\in V(\Y)$ and $s\in S_v$, $t\in S_w$. The local group at $(v,w)$ is $G_{(v,w)} = G_v\times G_w$. So for $s_{(v,w)}\in S_{\X\times\Y}$, $t_{(v,w)} \in S_{\X\times\Y}$ we have $s_{(v,w)}t_{(v,w)}=t_{(v,w)}s_{(v,w)}$. By the two previous statements this implies that for $v'\in V(\X)$ and $w'\in V(\Y)$ we have $s_{(v, w')}t_{(v',w)} = t_{(v',w)}s_{(v,w')}$.
		\item Let $v, v'\in V(\X)$, $w, w'\in V(\Y)$, $\alpha\in E(\X)$, $\beta\in E(\Y)$, $s\in S_w$, $t\in S_v$. Then $s_{(v,w)} = s_{(i(\alpha), w)} = s_{(t(\alpha), w)} = \psi_{(\alpha,w)}(s_{(i(\alpha,w))})$ and $(\alpha, w')^+ = (\alpha,w)^+$. So $(\alpha, w')^+s_{(v,w)}(\alpha, w')^- = s_{(v,w)}$ and hence $(\alpha, w')^+s_{(v,w)}= s_{(v,w)}(\alpha, w')^+ $. Similarly $(v', \beta)^+t_{(v,w)}(v', \beta)^- = t_{(v,w)}$ and hence $(v', \beta)^+t_{(v,w)} = t_{(v,w)}(v', \beta)^+ $.  
	\end{enumerate}

   We define the map $\phi:\pi_1({\cal G(X)}, T_{\X})\times \pi_1({\cal G(Y)}, T_{\Y})\rightarrow \pi_1({\cal G(X)\times G(Y)}, T_{\X\times\Y})$ on the generators. For $v\in V(\X)$, $s\in S_v$ set $\phi(s) = s_{(v, v_{\Y,0})}$ 
 and for $v\in V(\Y)$, $s\in S_v$ set $\phi(s) = s_{(v_{\X,0},v)}$. 
  For $\alpha \in E(\X)$ set $\phi(\alpha^+) = (\alpha, v_{\Y,0})^+$,  $\phi(\alpha^-) = (\alpha, v_{\Y,0})^-$ and for $\alpha\in E(\Y)$ set $\phi(\alpha^+) = (v_{\X,0}, \alpha)^+$, $\phi(\alpha^-) = (v_{\X,0}, \alpha)^-$. Note that the map $\phi$ induces a map from $E(\X)\sqcup E(\Y)\rightarrow E(\X\times\Y)$, which respects the composition and sends an edge in $T_{\X}\sqcup T_{\Y}$ to an edge in $T_{\X\times\Y}$. For $\alpha\in E(\X)$, $s\in S_{i(\alpha)}$ we have $\phi(\psi_{\alpha}(s)) = s_{(t(\alpha), v_{\Y,0})} = (\alpha, v_{\Y,0})^+s_{(i(\alpha), v_{\Y,0})}(\alpha, v_{\Y,0})^-= \phi(\alpha^+)\phi(s)\phi(\alpha^-)$ . Also for $v\in V(\X)$, we have $R_v\subset R_{(v, v_{\Y,0})}$. Similar statements hold when choosing $v\in V(\Y)$ and $\alpha\in E(\Y)$. So the relations in $\pi_1(\GX)$ and $\pi_1(\GY)$ are respected under $\phi$. By the consequences listed above the commutation relations are also satisfied. Indeed for $\alpha\in E(\X)$, $\beta\in E(\Y)$, $$\phi(\alpha^+\beta^+) = (\alpha, v_{\Y,0})^+(v_{\X,0},\beta)^+ = (v_{\X,0},\beta)^+(\alpha, v_{\Y,0})^+ =\phi(\beta^+\alpha^+).$$ For $v\in V(\X)$, $w\in V(\Y)$, $s\in S_v$, $t\in S_w$ we have $\phi(s)\phi(t) =  \phi(t)\phi(s)$ by statement 6 above. Finally for $\alpha\in E(\X)$, $w\in V(\Y)$, $t\in S_w$ we have $\phi(t)\phi(\alpha^+) = \phi(\alpha^+)\phi(t)$ by using statement 7. There is a corresponding equality for $\alpha\in E(\Y)$, $w\in V(\X)$, $t\in S_w$. So $\phi$ is a homomorphism.

We define the map $\xi: \pi_1({\cal G(X)\times G(Y)}, T_{\X\times\Y})\rightarrow \pi_1({\cal G(X)}, T_{\X})\times \pi_1({\cal G(Y)}, T_{\Y})$ on the generators. For $ s = s_{(v,w)}\in S_{\X\times\Y}$ set $\xi(s_{(v,w)}) = s\in S_v\sqcup S_w\subset S_{\X}\sqcup S_{\Y}$. For $\alpha\in E(\X)$, $w\in V(\Y)$, let $\xi((\alpha, w)^+) = \alpha^+$ and $\xi((\alpha, w)^-) = \alpha^-$. For $\beta\in E(\Y)$, $v\in V(\X)$, let $\xi((v, \beta)^+) = \beta^+$ and $\xi((v, \beta)^-) = \beta^-$. For $(\alpha, \beta)\in E(\X)\times E(\Y)$, let $\xi((\alpha,\beta)^+) = \alpha^+\beta^+$ and $\xi((\alpha,\beta)^-) = \beta^-\alpha^-$. Similarly to $\phi$ one can check that $\xi$ is a homomorphism.

Finally for $v\in V(\X)\sqcup V(\Y)$, $s\in S_v$ we have $\xi(\phi(s)) = s$. For $\alpha\in E(\X)\sqcup E(\Y)$, $\xi(\phi(\alpha^+)) = \alpha^+$ and $\xi(\phi(\alpha^-)) = \alpha^-$. For $v\in V(\X)$, $w\in V(\Y)$, $s\in S_v$, $t\in S_w$ we have $\phi(\xi(s_{(v,w)})) = \phi(s) = s_{(v, v_{\Y,0})} = s_{(v,w)}$ and $\phi(\xi(t_{(v,w)})) = t_{(v,w)}$. For $v\in V(\X)$, $w\in V(\Y)$, $\alpha\in E(\X)$, $\beta\in E(\Y)$ we have $\phi(\xi((v,\beta)^+)) = (v,\beta)^+$ and  $\phi(\xi((\alpha,\beta)^+)) = \phi(\alpha^+\beta^+) = (\alpha, v_{\Y,0})^+(v_{\X,0},\beta)^+ = (\alpha, \beta)^+$ and \linebreak $\phi(\xi((\alpha, w)^+)) = \phi(\alpha^+) = (\alpha, v_{\Y,0})^+ = (\alpha, w)^+$. So $\phi$ and $\xi$ are isomorphisms.

\end{proof}

\begin{Beispiel}
	In Example \ref{cyclic star} the fundamental group of $\D(\Y_S)$ is $\times_{v\in S}C_v$. In Example \ref{cyclic circle} the fundamental group of $\D(\calZ_S)$ is $\Z^{|S|}$. In Example \ref{Coxeter scwol} the fundamental group of $\mathfrak{W}(\Y_S)$ is the Coxeter group $W$.
\end{Beispiel}

We will now consider morphisms. Consider two scwols $\X$ and $\Y$. A morphism of scwols $f: \X\rightarrow \Y$ is a map that sends $V(\X)$ to $V(\Y)$, sends $E(\X)$ to $E(\Y)$ and such that:
\begin{enumerate}
	\item for every $\alpha\in E(\X)$, $f(i(\alpha)) = i(f(\alpha))$ and $f(t(\alpha)) = t(f(\alpha))$,
	\item for composable edges $\alpha, \beta\in E(\X)$, $f(\alpha\beta) = f(\alpha)f(\beta)$.
\end{enumerate}

Let $\G(\X) = (G_v, \psi_{\alpha})$ and $\calH(\Y) = (H_v, \xi_{\alpha})$ be two simple complexes of groups. Let $f:\X\rightarrow \Y$ be a morphism of scwols. A morphism of complexes of groups $\phi = (\phi_v, \phi(\alpha)): \G(\X)\rightarrow \calH(\Y)$ over $f$ consists of: 
\begin{enumerate}
	\item a homomorphism $\phi_v: G_v\rightarrow H_{f(v)}$ for every $v\in V(\X)$ and
	\item an element $\phi(\alpha)\in H_{t(f(\alpha))}$ for every edge $\alpha\in E(\X)$ such that \linebreak $\mathrm{Ad}(\phi(\alpha))\xi_{f(\alpha)}\phi_{i(\alpha)} = \phi_{t(\alpha)}\psi_{\alpha}$ and $\phi(\alpha\beta) = \phi(\alpha)\xi_{f(\alpha)}(\phi(\beta))$ for all composable edges $\alpha, \beta\in E(\X)$,
\end{enumerate}
where $\mathrm{Ad}(\phi(\alpha))$ is the conjugation by $\phi(\alpha)$.

\hspace{1pt}

Finally let $G$ be a group. A morphism $\phi = (\phi_v, \phi(\alpha)): \G(\X)\rightarrow G$ consists of a homomorphism $\phi_v: G_v\rightarrow G$ for each $v\in V(\X)$ and an element $\phi(\alpha)\in G$ for each $\alpha\in E(\X)$ such that $\mathrm{Ad}(\phi(\alpha))\phi_{i(\alpha)} = \phi_{t(\alpha)}\psi_{\alpha}$ and $\phi(\alpha\beta) = \phi(\alpha)\phi(\beta)$.
 There is always a morphism from the complex of groups to the fundamental group of the complex of groups.
\begin{Beispiel}\label{morphism cyclic star}
	 Consider the complex of groups $\D(\Y_S)$ given in Example \ref{cyclic star}. Its fundamental group is $\pi_1(\D(\Y_S)) = \times_{v\in S}C_v$.  
	For $A\in V(\Y_S)$, let \linebreak $\phi^S_A: \times_{v\in A}C_v\rightarrow \times_{v\in S}C_v$ with $\phi^S(g) = g$ and for $\alpha\in E(\Y_S)$ let $\phi^S(\alpha) = \neutre\in \times_{v\in S}C_v$. The morphism $\phi^S = (\phi^S_A,\phi^S(\alpha)): \D(\Y_S)\rightarrow \times_{v\in S}C_v$ is injective on the local groups.
\end{Beispiel}

\begin{Beispiel}\label{morphism Coxeter}
	Consider the complex of groups $\mathfrak{W}(\Y_S)$ given in Example \ref{Coxeter scwol}. Its fundamental group is $\pi_1(\mathfrak{W}(\Y_S)) = W$.  
	For $A\in V(\Y_S)$, let $\phi_A: W_A\rightarrow W$ with $\phi(g) = g$ and for $\alpha\in E(\Y_S)$ let $\phi(\alpha) = \neutre\in W$. The morphism \linebreak $\phi = (\phi_A,\phi(\alpha)): \mathfrak{W}(\Y_S)\rightarrow W$ is injective on the local groups.
\end{Beispiel}

\begin{Beispiel}\label{morphism cyclic circle}
	Consider the complex of groups $\D(\calZ_S)$ given in Example \ref{cyclic circle}. Its fundamental group is $\pi_1(\D(\calZ_S)) = \Z^{|S|}$. For the notation: let $\neutre$ be the trivial element in $\Z^{|S|}$ and let $x_s$, $s\in S$ be the standard generators of $\Z^{|S|}$.
	For $A\in V(\calZ_S)$, let $\phi^S_A: \langle \neutre\rangle \rightarrow \Z^{|S|}$ with $\phi^S(\neutre) = \neutre$ and for $(A, B,\lambda)\in E(\calZ_S)$ let $\phi^S((A, B,\lambda)) = \Pi_{s\in\lambda}x_s$. The morphism $\phi^S = (\phi^S_A,\phi^S(\alpha)): \D(\calZ_S)\rightarrow \Z^{|S|}$ is injective on the local groups.
\end{Beispiel}

\begin{Definition}
	A complex of groups $\G(\X)$ over a scwol $\X$ is developable if there exists a morphism $\phi$ from $\G(\X)$ to some group $G$ which is injective on the local groups.
\end{Definition}

\begin{Bemerkung}
	This definition is not the original definition given in III.$\C$.2.11\cite{BriHae} but it is equivalent to it by Corollary III.$\C$.2.15 in \cite{BriHae} and better suited to our use.
\end{Bemerkung}

Let $\G(\X)$ be a complex of groups over a scwol $\X$, let $G$ be a group and $\phi:\G(\X)\rightarrow G$ a morphism. The development of $\X$ with respect to $\phi$ is the scwol $\C(\X,\phi)$ given as follows:
\begin{enumerate}
	\item its vertices are $$V(\C(\X,\phi)) = \{(g\phi_v(G_v), v)\mid v\in V(\X), g\phi_v(G_v)\in G/\phi_v(G_v)\},$$
	\item its edges are $$E(\C(\X,\phi))=  \{(g\phi_{i(\alpha)}(G_{i(\alpha)}), \alpha) \mid \alpha\in E(\X), g\phi_{i(\alpha)}(G_{i(\alpha)})\in G/\phi_{i(\alpha)}(G_{i(\alpha)})\},$$
	\item the maps $i, t: E(\C(\X,\phi))\rightarrow V(\C(\X,\phi))$ are $$i(g\phi_{i(\alpha)}(G_{i(\alpha)}), \alpha) = (g\phi_{i(\alpha)}(G_{i(\alpha)}), i(\alpha))$$ and $$t(g\phi_{i(\alpha)}(G_{i(\alpha)}), \alpha) = (g\phi(\alpha)^{-1}\phi_{t(\alpha)}(G_{t(\alpha)}), t(\alpha)),$$
	\item the composition is $(g\phi_{i(\alpha)}(G_{i(\alpha)}), \alpha)(h\phi_{i(\beta)}(G_{i(\beta)}), \beta) = (h\phi_{i(\beta)}(G_{i(\beta)}), \alpha\beta)$, where $\alpha, \beta$ are composable and $g\phi_{i(\alpha)}(G_{i(\alpha)}) =h\phi(\beta)^{-1}\phi_{i(\alpha)}(G_{i(\alpha)})$.
\end{enumerate} 

Note that by Theorem III.$\C$.2.13 in \cite{BriHae}, $\C(\X,\phi)$ is indeed well-defined. Moreover there is an action of $G$ on $\C(\X,\phi)$ where $G\backslash \C(\X,\phi) = \X$.

As for simplicial complexes we can define geometric realizations of scwols. For a scwol $\X$ denote its geometric realization by $|\X|$. If a scwol does not have multiple edges, this construction coincides with the geometric realization of simplicial complexes. This is the only case we will need in this article and details on the general construction can be found in Chapter III.$\C$.1 in \cite{BriHae}. The action of $G$ on $\C(\X,\phi)$ induces an action of $G$ on $|\C(\X,\phi)|$. If we require the action of $G$ on $|\C(\X,\phi)|$ to be by isometries, putting a metric on $|\C(\X,\phi)|$ corresponds to putting a metric on $|\X|$ as $G\backslash |\C(\X,\phi)| = |\X|$.

\begin{Beispiel}
Consider the complex $\D(\Y_S)$ and the morphism \linebreak $\phi^S: \D(\Y_S)\rightarrow \Pi_{v\in S}C_v$ from Example \ref{morphism cyclic star}. One can check that the development $\C(\Y_S,\phi^S)$ is the product $\Pi_{v\in S}\C(\Y_{\{v\}},\phi^{\{v\}})$. Each $\C(\Y_{\{v\}},\phi^{\{v\}})$ is a scwol with set of vertices $\{(g, \emptyset) \mid g\in C_v\}\cup \{(C_v, \{v\})\}$ and set of edges $\{(g, e_v)\mid g\in C_v\}$ with $i(g, e_v) = (g, \emptyset)$ and $t(g, e_v) = (C_v, \{v\})$. So it is a star on $|C_v|$ branches, the tips correspond to the vertices $\{(g, \emptyset) \mid g\in C_v\}$, the central vertex is $(C_v, \{v\})$ and the edges are oriented from the tips to the center. The group $C_v$ acts by rotation and stabilizes the central vertex. For each $v\in S$ choose $l_v>0$. Let $\stern(v)$ be the geometric realization of $\C(\Y_{v}, \phi^{\{v\}})$ as follows: for $g\in C_v$ consider the interval $I_g = [0, l_v]$ then $\stern(v) = \bigcup_{g\in C_v}I_g/\sim$ where $0\in I_g\sim 0\in I_{\neutre}$. Note that $C_v$ acts by isometries on $\stern(v)$. 
The space $\stern(S) = \Pi_{v\in S}\stern(v)$ with the product metric is a geometric realization of $\C(\Y_S,\phi^S)$, due to the product structure of $\C(\Y_S,\phi^S)$. Moreover $\Pi_{v\in S}C_v$ acts by isometries on $\stern(S)$.

\end{Beispiel}

\begin{Beispiel}
	Consider the complex $\mathfrak{W}(\Y_S)$ and the morphism $\phi: \mathfrak{W}(\Y_S)\rightarrow W$ from Example \ref{morphism Coxeter}. The development $\C(\Y_S,\phi)$ is a scwol with set of vertices $\{(gW_A, A)\mid A\subset S, gW_A\in W/W_A\}$ and $\{(gW_A, (A, B))\mid A\subsetneq B,\ gW_A\in W/W_A\}$ as set of edges where $i(gW_A, (A,B)) = (gW_A, A)$ and $t(gW_A, (A,B)) = (gW_B, B)$. It is the scwol associated to the poset $W\calP(S) = \bigcup_{T\subset S}W/W_T$, the poset of parabolic cosets ordered by inclusion. In Section \ref{Coxeter} we will introduce the Coxeter polytope $C_W$ of $W$, which is a geometric realization of $\C(\Y_S,\phi)$. 
\end{Beispiel}

\begin{Beispiel}
	Consider the complex $\D(\calZ_S)$ and the morphism $\phi^S: \D(\calZ_S)\rightarrow \Z^{|S|}$ from Example \ref{morphism cyclic circle}. One can check that the development $\C(\calZ_S,\phi^S)$ is the product $\times_{v\in S}\C(\calZ_{\{v\}},\phi^{\{v\}})$. Each $\C(\calZ_{\{v\}},\phi^{\{v\}})$ is a scwol with vertices $$V(\C(\calZ_{\{v\}},\phi^{\{v\}}))=\{(g, \emptyset) \mid g\in \langle x_v\rangle\}\cup \{(g, \{v\}) \mid g\in \langle x_v\rangle\}$$ and edges $$E(\C(\calZ_{\{v\}},\phi^{\{v\}}))=\{(g, (\emptyset, \{v\},\emptyset ))\mid g\in \langle x_v\rangle\}\cup\{(g, (\emptyset, \{v\},\{v\} ))\mid g\in \langle x_v\rangle\}$$ where $i(g, (\emptyset, \{v\},\emptyset )) = (g, \emptyset)$, and $t(g, (\emptyset, \{v\},\emptyset )) = (g, \{v\})$, and \linebreak $i(g, (\emptyset, \{v\},\{v\} )) = (g, \emptyset)$, and $t(g, (\emptyset, \{v\},\{v\} )) = (gx_v^{-1}, \{v\})$. So it is a line, each vertex has either two incoming or two outgoing edges. The group $\Z= \langle x_v\rangle$ acts by translation. There are two orbits, one corresponding to the vertices with incoming edges, one to the vertices with outgoing edges. A geometric realization of $\C(\calZ_{\{v\}}, \phi^{\{v\}})$ is the real line $\R$, where we identify $(\neutre, \emptyset)\in \C(\Z_{\{v\}}, \phi^{\{v\}})$ with $0\in\R$, and $(\neutre, \{v\})$ with $0.5$ and $(x_v^{-1}, \{v\})$ with $-0.5$. Since we want $\Z$ to act by isometries, this means that for every $x_v^k\in\Z$ we identify the vertex $(x_v^k, \emptyset)$ with $k\in\R$ and $(x_v^k, \{v\})$ with $k+0.5\in\R$. Using the product structure we get that $\R^{|S|}$ with the Euclidean metric is a geometric realization of $\C(\calZ_S,\phi^S)$ on which $\Z^{|S|}$ acts by translation. 
\end{Beispiel}

%% file: Coxeterv2.tex
\subsection{Coxeter groups and the Davis-Moussong complex}\label{Coxeter}
The discussion of the Davis-Moussong complex is based primarily on \cite{Davis}. We will omit most proofs as they can be found in the literature, in particular in \cite{Davis} and \cite{Bourba}.

Let $S$ be a finite set. Let $M = (m(s,t))_{s,t\in S}$ be a symmetric matrix with $m(s,t)\in \N\cup\{\infty\}$, $m(s,s)=1$ and $m(s,t)=m(t,s)\geq 2$ if $s\neq t$. Such a matrix is called a Coxeter matrix. The Coxeter group associated to $M$ is given by the following presentation $$
W = \langle s\in S\mid (st)^{m(s,t)} = 1 \text{ for all } s, t \in S \rangle,
$$
 where $m(s,t)= \infty$ means there is no relation given between $s$ and $t$.
The pair $(W,S)$ is called a Coxeter system. Consider the $S\times S$ matrix $c$ defined by $c_{st} = \cos(\pi-\pi/m(s,t))$, the matrix $c$ is called the cosine matrix of the Coxeter matrix $M$. When $m(s,t)=\infty$, we intepret $\pi/\infty$ to be $0$ and $\cos(\pi-\pi/\infty) = -1$. The following fact states a classical result giving a necessary and sufficient condition for a Coxeter group to be finite.

\begin{Fact}[Theorem 6.12.9 \cite{Davis}]\label{spherical characterization}
	A Coxeter group $W$ is finite if and only if the cosine matrix $c$ is positive definite.
\end{Fact} 

For $T\subset S$ let $W_T$ be the subgroup of $W$ generated by $T$. Consider the poset of spherical subsets $\calS = \{T\subset S\mid W_T \text{ is finite } \}$ ordered by inclusion. In an abuse of notation, let us also write $\calS$ for the scwol associated to the poset $\calS$. Let $\mathfrak{W}(\calS)$ be the complex of groups over $\calS$ where the local group at $T\in\calS$ is $W_T$ and for an edge $(R,T)$ the associated map $\psi_{(R,T)}: W_R\rightarrow W_T$ is the inclusion $\psi_{(R,T)}(r) = r$ for every $r\in R$. The fundamental group of $\mathfrak{W}(\calS)$ is $W$ and there is an injective morphism $\phi = (\phi_T, \phi((R,T)))$ where $\phi_T: W_T\rightarrow W$ is the inclusion and $\phi((R,T)) = \neutre$ for every edge $(R,T)$. Let $\C(\calS,\phi)$ be the development of $\mathfrak{W}(\calS)$ with respect to $\phi$. Let us also consider the poset $W\calS = \bigcup_{T\in\calS}W/W_T$, called the poset of spherical cosets. In a similar abuse of notation, let us also write $W\calS$ for the scwol associated to the poset $W\calS$.

\begin{Bemerkung}\label{Coxeter scwols}
	The set of vertices of $\C(\calS,\phi)$ is $\{(wW_T, T)\mid T\in\calS, wW_T\in W/W_T\}$. The set of edges of $\C(\calS,\phi)$ is $\{(wW_R, (R,T))\mid R, T\in\calS,\ R\subset T,\ wW_R\in W/W_R\}$, where $(wW_R, (R,T))$ is an edge from the vertex $(wW_R, R)$ to the vertex $(wW_T,T)$. In particular there is an edge from a vertex $(wW_R,R)$ to a vertex $(w'W_T, T)$ if and only if $R\subset T$ and $wW_T = w'W_T$ (i.e. $w'^{-1}w \in W_T$). 
\end{Bemerkung}

\begin{Lemma}\label{Coxeter description}
	The scwols $\C(\calS,\phi)$ and $W\calS$ are equal.
\end{Lemma}
\begin{proof}
	It follows from Remark \ref{Coxeter scwols} that the two scwols have the same set of vertices. For the edges note that for $wW_R, w'W_T\in W\calS$, we have $wW_R\subset w'W_T$ if and only if $R\subset T$ and $w'^{-1}w\in W_T$. So using Remark \ref{Coxeter scwols}, the sets of edges also coincide.
\end{proof}

\paragraph{Coxeter polytope} From now on, assume $W$ is finite. Let us recall the canonical representation of $W$. Consider $\Pi = \{\alpha_s \mid s\in S\}$ and $V = \oplus_{s\in S} \R\alpha_s$. Let $\langle\cdot,\cdot\rangle : V\times V\rightarrow \R$ be the scalar product on $V$ given by $\langle\alpha_s, \alpha_t\rangle = -\cos(\frac{\pi}{m(s,t)})$. The canonical representation of $W$ on $GL(V)$ is given by $\rho: W \rightarrow GL(V)$ with $\rho(s)(x) = x - 2\langle \alpha_s, x\rangle\alpha_s$, for $s\in S$, $x\in V$. The scalar product on $V$ is $\rho(W)$-invariant. There is a dual basis $\Pi^*= \{\alpha_s^* \mid s\in S\}$, satisfying $\langle\alpha_s^*,\alpha_t\rangle= 0$ if $s\neq t$ and $\langle\alpha_s^*,\alpha_s\rangle= 1$. We choose $x_0 = \sum_{s\in S}\alpha_s^* \in V$. The Coxeter polytope of $(W,S)$, denoted $\Cox(W)$, is the convex hull of $\{\rho(w)(x_0)\in V\mid w\in W\}$. It is endowed with the Euclidean metric. Note that its interior is non-empty. For a subset $T\subset S$, let $\Pi_T = \{\alpha_s \mid s\in T\}$ and $V_T$ be the subvector space of $V$ spaned by $\Pi_T$. Let $\Pi_T^*= \{\alpha_{s, T}^* \mid s\in T\}$ be the dual basis of $\Pi_T$ in $V_T$. Fix $x_{0,T} = \sum_{s\in T}\alpha_{s, T}^*$. Let $\Cox(W_T)$ be the convex hull of $\{\rho(w)(x_{0, T})\in V_T\mid w\in W_T\}$. Moreover let $\Cox_T(W)$ be the convex hull of $\{\rho(w)(x_0)\in V\mid w\in W_T\}$. Let $u = x_0-x_{0,T}$ and $t_u: V\rightarrow V$ be the translation by the vector $u$. This translation sends $\rho(w)(x_{0, T})$ to $\rho(w)(x_{0})$ for every $w\in W_T$. Specifically it is an isometry from $\Cox(W_T)$ to $\Cox_T(W)$.

\begin{Lemma}[Lemma 7.3.3 \cite{Davis}]
	The poset $W\calS$ and the face poset $\F(\Cox(W))$ of $\Cox(W)$ are isomorphic. Specifically the correspondance $wW_T\rightarrow w\Cox_T(W)$ induces an isomorphism of posets. 
\end{Lemma}

So we can identify $W\calS$ and hence $\C(\calS, \phi)$ with the barycentric subdivision of the Coxeter polytope $\Cox(W)$. Thus identifying $|\C(\calS, \phi)|$ isometrically with $\Cox(W)$. The metric on $|\C(\calS, \phi)|$ induced by the identification with $\Cox(W)$ is called the Moussong metric. In particular for $wW_T\in W\calS$ the geometric realisation $|W\calS_{\leq wW_T}|\subset |W\calS|$ is identified with the face $w\Cox_T(W)$.

\paragraph{The General Case}
We now consider any Coxeter group $W$, so $W$ need not necessarily be finite. We put a coarser cell structure on $W\calS$ (or equivalently on $\C(\calS,\phi)$) to build the Davis-Moussong complex $\Sigma$ by identifying each subposet $(W\calS)_{\leq wW_T, T\in\calS}$, which is isomorphic to the poset $W_T(\calS_{\leq T})$, with a Coxeter polytope $\Cox(W_T)$. So we can give the following description of $\Sigma$.

\begin{Theorem}[\cite{Davis} Proposition 7.3.4.]\label{Davis} There is a natural cell structure on $\Sigma$ so that \begin{enumerate}
		\item its vertex set is $W$, its 1-skeleton is the Cayley graph $\mathrm{Cay}(W,S)$, and its 2-skeleton is a Cayley 2-complex,
		\item each cell is a Coxeter polytope,
		\item the link $\Lk(v,\Sigma)$ of each vertex is isomorphic to the abstract simplicial complex $\calS_{>\emptyset}$,
		\item a subset of $W$ is the vertex set of a cell if and only if it is a spherical coset,
		\item the poset of cells in $\Sigma$ is $W\calS$. 
	\end{enumerate}
\end{Theorem}

Note that the Cayley graph $\mathrm{Cay}(W,S)$ is considered to be undirected, hence there are no double edges between vertices, even though all elements of $S$ have order 2 in $W$. Furthermore all edges in $\mathrm{Cay}(W,S)$ are labeled, hence the edges of $\Sigma$ are labeled. This labeling coincides with the labeling of vertices in $\Lk(v,\Sigma)$.

Now that we have an appropriate description of $\Sigma$, let us state the following geometric property of $\Sigma$.

\begin{Theorem}[Moussong's Theorem \cite{Mousso}]
	For any Coxeter system the associated cell complex $\Sigma$, equipped with its natural piecewise Euclidean metric, is CAT(0).
\end{Theorem}

 A simplicial complex $L$ with piecewise spherical structure has simplices of size $\geq \pi/2$ if each of its edges has length $\geq \pi/2$. Such a simplicial complex is a metric flag complex if the following condition holds: Suppose $\{v_0,\dots, v_k\}$ is a set of pairwise adjacent vertices of $L$. Put $c_{ij} = \cos(d(v_i,v_j))$. Then $\{v_0,\dots, v_k\}$ spans a simplex if and only if the matrix $(c_{ij})$ is positive definite. Then Moussong's Theorem is the consequence of the following lemmata.

\begin{Lemma}[\cite{Davis} Lemma 12.3.1.]\label{posdef metric}
	Let $\Lk$ be the link of a vertex in $\Sigma$ with its natural piecewise spherical structure inherited from $\Sigma$. Then $\Lk$ is a simplicial complex and has simplices of size $\geq \pi/2$. Moreover, it is a metric flag complex. 
\end{Lemma} 

Note that using Fact \ref{spherical characterization} the set of vertices of $\Lk$ is $S$ and $T\subset S$ spans a simplex if and only if $W_T$ is finite. Moreover the distance between two vertices in $\Lk$ is given by $d(v,w) = \pi-\pi/m(v,w)$.
 
\begin{Lemma}[Moussong's Lemma \cite{Davis} Lemma I.7.4.]\label{Moussong}
	Suppose $L$ is piecewise spherical simplicial cell complex in which all cells are simplices of size $\geq \pi/2$. Then $L$ is CAT(1) if and only if it is a metric flag complex.
\end{Lemma}

%% file: Salvetti.tex
\subsection{Right-angled Artin groups and the Salvetti complex}\label{sec:Salvetti}
Every Coxeter group has an associated Artin group. We will concentrate on the class of right-angled Artin groups and present their analog to the Davis-Moussong complex, the Salvetti complex $S_{\Gamma}$. An extensive discussion of right-angled Artin groups can be found in Charney's survey \cite{Charne}.

Given a simplicial graph $\Gamma$, with vertex set $V$ and edge set $E$, the associated right-angled Artin group $A(\Gamma)$ is given by the following presentation $$A(\Gamma)=\langle x_v, v\in V\mid \text{ for every } e=\{v, w\}\in E, \ x_vx_w=x_wx_v \rangle.$$

If $\Gamma$ has no edges $A(\Gamma)$ is the free group of rank $|V|$, if $\Gamma$ is a complete graph $A(\Gamma)$ is the free abelian group of rank $|V|$.

\paragraph{Salvetti complex $S_{\Gamma}$} Let $\Gamma$ be a simplicial graph with vertex set $V$. For any set of pairwise adjacent vertices $V' = \{v_1, \dots, v_n\}$ consider the corresponding generators $x_i=x_{v_i}$ and let $$C(V') = \{ x_1^{\epsilon_1}\cdots x_n^{\epsilon_n}\in A(\Gamma)\mid \epsilon_i \in\{0,1\} \text{ for every } i\in \{1,\dots,n\}\}.$$ Note that for two sets $V', V''\subset V$ of pairwise adjacent vertices we have $V'\neq V''$ if and only if $C(V')\neq C(V'')$. The Salvetti complex $S_{\Gamma}$ is a cube complex with vertex set $A_{\Gamma}$ and the cubes are sets of vertices of the form $aC(V')$ for some $a\in A(\Gamma)$ and $V'\subset V$ a set of pairwise adjacent vertices of $\Gamma$. The Salvetti complex $S_{\Gamma}$ is known to be a CAT(0) cube complex by \cite{ChaDav}.

%% file: scwolv3.tex
\section{The cell complex $\Sigma$}\label{Sigma}
The goal of this section is to show geometrically that Dyer groups are CAT(0) by constructing an appropriate Euclidean cell complex $\Sigma$. The first step is to construct a scwol $\C$ associated to a Dyer group. The scwol $\C$ encodes the necessary information to build $\Sigma$. The vertices of $\C$ will correspond to subcomplexes of $\Sigma$ and the edges of $\C$ will encode identifications between subcomplexes of $\Sigma$. Finally we will also be able to interpret $\C$ as a simplicial subdivision of the complex $\Sigma$. We will first focus on spherical Dyer groups $D$ which factor as a direct product of a finite Coxeter group and cyclic groups. We start with the construction of a scwol $\X$ associated to a spherical Dyer group $D$ and then define a complex of groups $\D(\X)$. The scwol $\C$ will be the development of the complex of groups $\D(\X)$. The second subsection will discuss this for general Dyer groups. The third subsection will be devoted to the Euclidean cell complex $\Sigma$.
\subsection{A combinatorial structure for spherical Dyer groups}\label{sec:spherical}
A Dyer group $D = D(\Gamma, f, m)$ is \textit{spherical} if its underlying graph $\Gamma$ is complete and the subgroup $D_2$ is a finite Coxeter group. If $D = D(\Gamma, f, m)$, we also say that $\Gamma=(\Gamma, f, m)$ is a spherical Dyer graph. In this section we will assume $D$ is a spherical Dyer group. In particular we then have $D=D_2\times D_p\times D_{\infty}$, where $D_2$ is a finite Coxeter group, $D_p$ is a direct product of finite cyclic groups and $D_{\infty} = \Z^{|V_{\infty}|}$. As with Coxeter groups, we can characterize spherical Dyer groups through the cosine matrix. Let $(\Gamma, f, m)$ be a Dyer graph and let $V = V(\Gamma)$ and $E = E(\Gamma)$. We extend the map $m: E\rightarrow \N_{\geq 2}$ to a map $m: V\times V\rightarrow \N_{\geq 2}\cup\{\infty\}$ by setting $m(u, v) = m(\{u,v\})$ if $\{u,v\}\in E$, and $m(u,v) = \infty$ if $u\neq v$ and $\{u,v\}\notin E$, and $m(u,u) = 1$. We interpret $\pi/\infty$ to be $0$ and $\cos(\pi-\pi/\infty) = \cos(\pi) = -1$. The cosine matrix associated to a Dyer graph $(\Gamma, f, m)$ is the $V\times V$ matrix $c$ defined by $c_{uv} = \cos(\pi-\pi/m(u,v))$. 
 The following characterization of spherical Dyer groups follows from Fact \ref{spherical characterization}. 

\begin{Lemma}\label{posdef}
	A Dyer group $D(\Gamma,f,m)$ is spherical if and only if the cosine matrix $c$ associated to $(\Gamma,f,m)$ is positive definite.
\end{Lemma}
\begin{proof}
	Assume $D$ is a spherical Dyer group. Then the restriction of $c$ to $V_2\times V_2$ is positive definite. Since additionally $\Gamma$ is a complete Dyer graph, this implies that the matrix $c$ is positive definite.
	Now assume the cosine matrix $c$ associated to $(\Gamma,f,m)$ is positive definite. 
	Consider the matrix $M = (m(u,v))_{u,v\in V}$. Then the cosine matrix $c$ associated to $(\Gamma,f,m)$ is equal to the cosine matrix of the Coxeter matrix $M$ as defined in Section \ref{Coxeter}. 
	So by Fact \ref{spherical characterization} the cosine matrix $c$ is positive definite if and only if the Coxeter group associated to $M$ is finite. 
	So we have $m(u,v)\neq\infty$ for all $u,v\in V$. Moreover since $\Gamma$ is a Dyer graph this also implies that the restriction of $c$ to $V_2\times V_2$ is positive definite. So the graph $\Gamma$ is complete and $D_2$ is a finite Coxeter group by Fact \ref{spherical characterization}. Hence $D$ is a spherical Dyer group.

\end{proof}

Let $\X = \X(\Gamma)$ be the scwol with set of vertices $V(\X) =  \{X\subset V\}$ and set of edges $E(\X) = \{ (X,Y,\omega) \mid X \subsetneq Y\subset V(\Gamma),\ \omega\subset (Y\setminus X)_{\infty} \}$ with $i(X,Y,\omega) = X$ and $t(X,Y,\omega) = Y$ and $(Y,Z,\omega')(X,Y,\omega) = (X,Z, \omega\cup\omega')$. We call $\X$ the scwol associated to the spherical Dyer graph $\Gamma$. Similarly to the group $D$ we can also describe $\X$ as a direct product of scwols.

\begin{Lemma}\label{product scwol}
	Let $\X_2 = \X(\Gamma_2)$, $\X_p = \X(\Gamma_p)$ and $\X_{\infty} = \X(\Gamma_{\infty})$. Then we have the product decomposition $\X = \X_2\times\X_p\times\X_{\infty}$. Moreover $\X_p = \Y_{V_p} =\times_{v\in V_{p}}\Y_{v}$ and $\X_{\infty} = \calZ_{V_{\infty}} = \times_{v\in V_{\infty}}\calZ_{v}$ as in Examples \ref{star} and \ref{circle}.
\end{Lemma}
\begin{proof}
Since $V = V_2\sqcup V_p\sqcup V_{\infty}$, every $X\in V(\X)$ can be decomposed as a disjoint union $X = X_2\sqcup X_p\sqcup X_{\infty}$, so $V(\X) = V(\X_2)\times V(\X_p)\times V(\X_{\infty})$. For the edges note that $(X,Y,\omega) \in E(\X)$ if and only if $X_i\subset Y_i$ for every $i\in \{2,p,\infty\}$ and at least one of those inclusions is strict and $\omega\subset Y_{\infty}\setminus X_{\infty}$.
\end{proof}

We define the simple complex of groups $\D(\X)$ over the scwol $\X$. For each $X\in V(\X)$ let the local group be $D^f_X = D_{X_2\cup X_p}$. As mentioned in Remark \ref{SubDyer}, we know by \cite{Dyer} that if $X\subset Y$,  $D^f_X<D^f_Y<D$. For each edge $(X,Y,\omega)\in E(\X)$, let $\psi_{(X,Y,\omega)}: D^f_X\rightarrow D^f_Y$ be the map induced by $\psi(x_v)=x_v$ for every $v\in X_2\cup X_p$. These maps are all injective. Note that they do not depend on $\omega$. We also introduce the morphism $\phi = \phi^{\Gamma}: \D(\X)\rightarrow D$ where $\phi_X = \phi^{\Gamma}_X: D^f_X\rightarrow D$ is the natural inclusion and $\phi(X,Y,\omega) = \phi^{\Gamma}(X,Y,\omega) = \Pi_{v\in\omega}x_v$. Note that $\phi(X,Y,\omega)$ is well-defined since the subgraph spanned by $\omega\subset V_{\infty}$ is complete. Moreover $\phi(X,Y,\omega)$ only depends on $\omega$ so we will write $\phi(X,Y,\omega) = \phi(\omega) = \Pi_{v\in\omega}x_v$. Also note that each local group $D^f_X$ is finite. 

\begin{Bemerkung}\label{product complex} For every $X\in V(\X)$ the local group $D^f_X$ can be decomposed as $D^f_X = D^f_{X_2}\times D^f_{X_p}$. As $D^f_{X_{\infty}}$ is the trivial group we also have $D^f_X \cong D^f_{X_2}\times D^f_{X_p}\times D^f_{X_{\infty}}$. So using Lemma \ref{product scwol}, we have that the complex $\D(\X)$ is isomorphic to the product $\D(\X_2)\times\D(\X_p)\times\D(\X_{\infty})$. Moreover $\D(\X_p)$ is isomorphic to $\Pi_{v\in V_p}\D(\Y_v)$ and $\D(\X_{\infty})$ is isomorphic to $\Pi_{v\in V_{\infty}}\D(\calZ_v)$. The morphism $\phi = \phi^{\Gamma}$ also decomposes as a product $\phi = \phi_2\times\phi_p\times\phi_{\infty}$ where $\phi_2 = \phi^{\Gamma_2}$, $\phi_p = \phi^{\Gamma_p}$ and $\phi_{\infty} = \phi^{\Gamma_{\infty}}$.
\end{Bemerkung}

\begin{Lemma}
	The fundamental group of $\D(\X)$ is $D$ and the complex of groups $\D(\X)$ is developable.
\end{Lemma}
\begin{proof}
We use the product decomposition given in Remark \ref{product complex}. The scwol $\X_2$ is associated to the poset $\calP(V_2)$ so it is simply connected. Moreover it contains a unique maximal element $V_2$ so the fundamental group of $\D(\X_2)$ is $D^f_{V_2} = D_2$. The same argument implies that the fundamental group of $\D(\X_p)$ is $D_p$. Recall that $\D(\X_{\infty})$ is isomorphic to $\Pi_{v\in V_{\infty}}\D(\calZ_v)$. The fundamental group of each $\D(\calZ_v)$ is $\Z$. So the fundamental group of $\D(\X_{\infty})$ is $\Z^{|V_{\infty}|} = D_{\infty}$. So the fundamental group of $\D(\X)$ is $D_2\times D_p\times D_{\infty} = D$. Since the maps $\phi_X$ are injective for all $X\in V(\X)$, the complex $\D(\X)$ is developable.
\end{proof}

\begin{Bemerkung} Since the complex $\D(\X)$ is developable we can describe its development $\C = \C(\X,\phi)$. Since $D^f_X<D$ and the maps $\phi_X$ are canonical inclusions, we will identify the image $\phi_X(D^f_X)$ with $D^f_X<D$. The set of vertices of $\C$ is $V(\C) = \{(gD^f_X, X)\mid X\in V(\X),\ gD^f_X\in D/D^f_X\}$. The set of edges of $\C$ is $E(\C) = \{(gD^f_X, (X,Y,\omega))\mid (X,Y,\omega)\in E(\X),\ gD^f_X\in D/D^f_X\}$ where $i(gD^f_X, (X,Y,\omega)) = (gD^f_X, X)$ and $t(gD^f_X, (X,Y,\omega)) = (g\phi(\omega)^{-1}D^f_Y, Y)$. For a simpler notation we write $gX$ for a vertex $(gD^f_X,X)$ and $g(X,Y,\omega)$ for an edge $(gD^f_X, (X,Y,\omega))$. Note that $gX = hY$ if and only if $X=Y$ and $g^{-1}h\in D^f_X$. Similarly $g(X,Y,\omega) = h(X',Y', \omega')$ if and only if $X'= X$, $Y' = Y$, $\omega' = \omega$ and $g^{-1}h\in D^f_X$. In particular $\X$ is the quotient of $\C$ by the action of the group $D$.
\end{Bemerkung}

\begin{Lemma}\label{product development}
	The development $\C(\X,\phi)$ has a product decomposition \\ $\C(\X_2,\phi_2)\times \C(\X_p,\phi_p)\times \C(\X_{\infty},\phi_{\infty})$.
\end{Lemma}
\begin{proof}
This follows from the product decomposition of $\X$, $\D(\X)$, $D$ and $\phi$.
\end{proof}

\begin{Bemerkung}
	For $i\in\{2, p, \infty\}$, Lemma \ref{product development} implies that we can consider each scwol $\C(\X_i,\phi_i)$ to be a subscwol of $\C(\X,\phi)$. There is a canonical inclusion by identifying a vertex $gX\in \C(\X_i,\phi_i)$ with $gX\in\C(\X,\phi)$. The subscwol $\C(\X_i,\phi_i)$ is then stable under the action of $D_i$.
\end{Bemerkung}

\paragraph{Links and stars in $\C$}
For a well-defined construction of the cell complex $\Sigma$ we need to understand the local combinatorial structure of $\C$. As edges in $\C(\X,\phi)$ are oriented we will distinguish the incoming and the outgoing link and star of a vertex. Let $gY\in V(\C)$. The incoming link $\Lk_{in}(gY, \C)$ is the full subscwol of $\C$ spanned by the vertices $\{hZ \mid \exists e\in E(\C): t(e) = gY \text{ and } i(e) = hZ \}$. Similarly the outgoing link $\Lk_{out}(gY, \C)$ is the full subscwol of $\C$ spanned by the vertices $\{hZ \mid \exists e\in E(\C): i(e) = gY \text{ and } t(e) = hZ \}$.
The incoming star is the subscwol spanned by $gY$ and its incoming link so it is the combinatorial join $$\St_{in}(gY, \C) = \{gY\}\star\Lk_{in}(gY, \C)$$ and the outgoing star is defined similarly $$\St_{out}(gY, \C) = \{gY\}\star\Lk_{out}(gY, \C).$$

\begin{Bemerkung}
	$\St_{in}(gY, \C)$ is isomorphic to $\St_{in}(\neutre Y, \C(\X(\Gamma_Y),\phi^{\Gamma_Y}))$. Moreover the product decomposition of $\C$ induces a product decomposition of the incoming star $$\St_{in}(\neutre Y, \C) = \St_{in}(\neutre Y_2, \C(\X_2,\phi_2))\times\St_{in}(\neutre Y_p, \C(\X_p,\phi_p))\times\St_{in}(\neutre Y_{\infty}, \C(\X_{\infty},\phi_{\infty}))$$ and as such also a product decomposition for every $\St_{in}(gY, \C)$. Moreover for a vertex $hZ\in \St_{in}(gY, \C)$, the star $\St_{in}(hZ, \C)$ is a subscwol of $\St_{in}(gY, \C)$.
\end{Bemerkung}

\subsection{A combinatorial structure for general Dyer groups}\label{sec:general}

Let us now give a similar construction with analogous results for general Dyer groups. Let $(\Gamma, f, m)$ be a Dyer graph and $D=D(\Gamma)$ be the associated Dyer group. We note $V = V(\Gamma)$. Let $\X = \X(\Gamma)$ be the scwol with set of vertices $V(\X) =  \{X\subset V \mid D(\Gamma_X) \text{ is a spherical Dyer group}\}$ and set of edges $E(\X) = \{ (X,Y,\omega) \mid X, Y\in V(\X),\ X \subsetneq Y,\ \omega\subset (Y\setminus X)_{\infty} \}$ with $i(X,Y,\omega) = X$ and $t(X,Y,\omega) = Y$ and $(Y,Z,\omega')(X,Y,\omega) = (X,Z, \omega\cup\omega')$. The main difference with the spherical case, is the set of vertices of $\X$. Indeed we do not consider all subsets $X\subset V$ but only those for which $\Gamma_X$ is complete and the group $D^f_X = D_{X_2\cup X_p}$ is finite. We also define a complex of groups $\D(\X)$ over $\X$. For each $X\in V(\X)$ let the local group be $D^f_X = D_{X_2\cup X_p}$ and for each edge $(X,Y,\omega)\in E(\X)$, let $\psi_{(X,Y,\omega)}: D^f_X\rightarrow D^f_Y$ be the natural inclusion. By \cite{Dyer}, these maps are all injective. The local groups are all finite. We also introduce the morphism $\phi: \D(\X)\rightarrow D$ where $\phi_X: D^f_X\rightarrow D$ is the natural inclusion and $\phi(X,Y,\omega) = \phi(\omega) = \Pi_{v\in\omega}x_v$ (this element is well defined since $\omega\subset V_{\infty}$ and $\Gamma_{\omega}$ is complete). As in the spherical case we can write $\D(\X(\Gamma))$ and $\phi^{\Gamma}$ when also considering the same construction on a subgraph.
As before we are interested in the development of the complex of groups $\D(\X)$ so we first show that $\D(\X)$ is developable.

\begin{Beispiel}
	Consider the Dyer graph $\Gamma_{m,p}$, given again in Figure \ref{fig1rep}, and the Dyer group $D_{m,p}$ from Example \ref{nontrivial example}. The associated scwol $\X_{m,p}$ is drawn in Figure \ref{exscwol}. Its vertex set is $V(\X_{m,p}) = \{\emptyset, \{a\}, \{b\}, \{c\}, \{d\}, \{a,b\}, \{b,c\}, \{c, d\}\}$.
\end{Beispiel}

\begin{figure}[h]
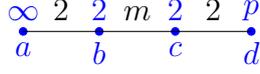

	\centering
	\include{graph1}
	\caption{Dyer graph $\Gamma_{m,p}$ for some $m, p\in\N_{\geq 2}$ as given in Figure \ref{fig1}}
	\label{fig1rep}
\end{figure}

\begin{figure}[h]
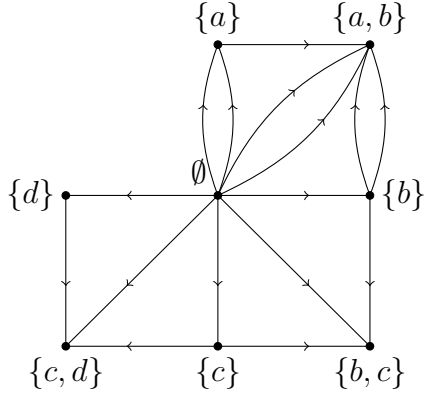

	\centering
	\include{exscwol}
	\caption{The scwol $\X_{m,p}$ associated to the graph $\Gamma_{m,p}$ given in Figure \ref{fig1rep}.}
	\label{exscwol}
\end{figure}

\begin{Lemma}\label{metric}
	The scwol $\X$ is isomorphic to the union of scwols $\Y = \bigcup_{Y\in V(\X)} \X_Y$, where $\X_Y$ is the scwol associated to the spherical Dyer group $D_Y$. The fundamental group of $\D(\X)$ is $D$. In particular, the complex of groups $\D(\X)$ is developable.
\end{Lemma}

\begin{proof}
	First we compare the sets of vertices. If $Y\in V(\X)$ then $Y\in V(\X_Y)$ so $Y\in V(\Y)$. On the other hand if $Y\in V(\Y)$, we have $Y\in V(\X_Z)$ for some $Z\in V(\X)$ so $Y\subset Z$ so $D_Y$ is spherical. This implies $V(\X) = V(\Y)$. Now we compare the sets of edges. If $e = (X, Y, \omega) \in E(\X)$ then $e\in E(\X_Y)$ so $e\in E(\Y)$. On the other hand if $e \in V(\Y)$, $e\in V(\X_Z)$ for some $Z\in V(\X)$ so $e = (X, Y, \omega)$ with $X\subsetneq Y\subset Z$ and $\omega \subset(Y\setminus X)_{\infty}$ so $e\in E(\X)$. This implies $E(\X) = E(\Y)$.
	We can now apply Seifert-van Kampen to $\Y$. The set $V(\X)$ is finite and each scwol $\X_Y$ is connected. We have $\emptyset\in V(\X_Y)$ for all $Y\in V(\X)$ and $\emptyset$ is adjacent to any vertex in any $\X_Y$. So $\bigcap_{Y\in V(\X)} \X_Y$ is nonempty and connected. We can then use the presentations to see that the fundamental group of $\D(\X)$ is $D$. Finally, by \cite{Dyer}, the maps $\phi_X: D^f_X\rightarrow D$ are all injective. So $\D(\X)$ is developable.
\end{proof}

\begin{Bemerkung}
	Since the complex $\D(\X)$ is developable we can describe its development $\C = \C(\X,\phi)$. Since $D^f_X<D$ and the maps $\phi_X$ are canonical inclusions, we will identify the image $\phi_X(D^f_X)$ with $D^f_X$. The set of vertices of $\C$ is $V(\C) = \{(gD^f_X, X)\mid X\in V(\X),\ gD^f_X\in D/D^f_X\}$. The set of edges of $\C(\X, \phi)$ is $E(\C) = \{(gD^f_X, (X,Y,\omega))\mid (X,Y,\omega)\in E(\X),\ gD^f_X\in D/D^f_X\}$ where \linebreak $i(gD^f_X, (X,Y,\omega)) = (gD^f_X, X)$ and $t(gD^f_X, (X,Y,\omega)) = (g\phi(\omega)^{-1}D^f_Y, Y)$. For a simpler notation we write $gX$ for a vertex $(gD^f_X,X)$ and $g(X,Y,\omega)$ for an edge $(gD^f_X, (X,Y,\omega))$. Note that $gX = hY$ if and only if $X=Y$ and $g^{-1}h\in D^f_X$. Similarly $g(X,Y,\omega) = h(X',Y', \omega')$ if and only if $X'= X$, $Y' = Y$, $\omega' = \omega$ and $g^{-1}h\in D^f_X$. As in the spherical case $\C$ does not have multiple edges between two vertices.
\end{Bemerkung}

\paragraph{Links and stars in $\C$} As before we need to understand the local combinatorial strucuture of $\C$. As edges in $\C$ are oriented we will distinguish the incoming and the outgoing link and star of a vertex. We recall the definitions here. Let $gY\in V(\C)$. The incoming link $\Lk_{in}(gY, \C)$ is the full subscwol of $\C$ spanned by the vertices $\{hZ \mid \exists e\in E(\C): t(e) = gY \text{ and } i(e) = hZ \}$. Similarly the outgoing link $\Lk_{out}(gY, \C)$ is the full subscwol of $\C$ spanned by the vertices $\{hZ \mid \exists e\in E(\C): i(e) = gY \text{ and } t(e) = hZ \}$.
The incoming star is the subscwol spanned by $gY$ and its incoming link so it is the combinatorial join $$\St_{in}(gY, \C) = \{gY\}\star\Lk_{in}(gY, \C)$$ and the outgoing star is defined similarly $$\St_{out}(gY, \C) = \{gY\}\star\Lk_{out}(gY, \C).$$

\begin{Lemma}
	For every vertex $gY\in V(\C)$, the scwols $\St_{in}(gY, \C(\X,\phi))$ and $\St_{in}(\neutre Y, \C(\X_Y,\phi^{\Gamma_Y}))$ are isomorphic.
\end{Lemma}
\begin{proof}
	It suffices to show this for $g=\neutre$. Then the statement is clear as it follows directly from the property of the definitions.
\end{proof}

\input{CellComplex}

%% file: exscwol.tex
\begin{tikzpicture}[scale=2]
\coordinate (empty) at (0,0) ;
\coordinate (a) at (0,1) ;
\coordinate (b) at (1,0) ;
\coordinate (c) at (0,-1) ;
\coordinate (d) at (-1,0) ;
\coordinate (aetb) at (1,1) ;
\coordinate (betc) at (1,-1) ;
\coordinate (cetd) at (-1,-1) ;

\draw[fill] (empty) circle [radius=0.025]node[above left]{$\emptyset$};
\draw[fill] (a) circle [radius=0.025] node[above]{$\{a\}$};
\draw[fill] (b) circle [radius=0.025]node[right]{$\{b\}$};
\draw[fill] (c) circle [radius=0.025]node[below]{$\{c\}$};
\draw[fill] (d) circle [radius=0.025]node[left]{$\{d\}$};
\draw[fill] (aetb) circle [radius=0.025]node[above]{$\{a,b\}$};
\draw[fill] (betc) circle [radius=0.025]node[below]{$\{b,c\}$};
\draw[fill] (cetd) circle [radius=0.025]node[below]{$\{c,d\}$};

\draw[->-] (empty)--(d);
\draw[->-] (empty)--(c);
\draw[->-] (empty)--(b);
\draw[->-] (empty)--(cetd);
\draw[->-] (empty)--(betc);
\draw[->-] (d)--(cetd);
\draw[->-] (c)--(cetd);
\draw[->-] (c)--(betc);
\draw[->-] (b)--(betc);
\draw[->-] (a)--(aetb);
\draw[->-] (empty)to[bend left=20](a);
\draw[->-] (empty)to[bend right=20](a);
\draw[->-] (empty)to[bend left=20](aetb);
\draw[->-] (empty)to[bend right=20](aetb);
\draw[->-] (b)to[bend left=20](aetb);
\draw[->-] (b)to[bend right=20](aetb);

\end{tikzpicture}

%% file: CellComplex.tex
\subsection{The piecewise Euclidean cell complex $\Sigma$}\label{sec:sigma}
The scwol $\C$, which is also a simplicial complex, described in the previous section is a combinatorial object. To build the cell complex $\Sigma$ we could try to endow the geometric realization of $\C$ with a CAT(0) metric. This would give a simplicial complex with a non-standard piecewise Euclidean metric. The problem is that Moussong's Lemma \ref{Moussong} does not apply directly to simplicial complexes with a piecewise Euclidean metric since dihedral angles should be at least $\pi/2$. The idea is to interpret $\C$ as some generalized face poset of $\Sigma$. Indeed $\C$ does not give us the face structure of $\Sigma$ but some form of subcomplex structure. Each vertex in $\C$ corresponds to a  subcomplex of $\Sigma$ and edges give identifications between these subcomplexes. Nevertheless we will be able to interpret $\C$ as a simplicial subdivision of $\Sigma$. We start with the description and study of the subcomplexes associated to vertices, then build $\Sigma$ and finally show that $\Sigma$ is CAT(0) using Moussong's Lemma \ref{Moussong}.

Let $(\Gamma, f, m)$ be a Dyer graph, $D = D(\Gamma,f, m)$ the associated Dyer group, $\X = \X(\Gamma)$ the associated scwol and $\D(\X)$ the associated complex of groups. Consider the injective morphism $\phi : \D(\X)\rightarrow D$ given by the natural inclusion $\phi_X : D^f_X\rightarrow D$ and $\phi(X,Y,\omega)= \phi(\omega) = \Pi_{v\in\omega}x_v$. As in the previous section we construct the development $\C = \C(\X,\phi)$.

\paragraph{Elementary building blocks} Let $Y\in V(\X)$. First we consider elementary building blocks in the cases $Y=Y_2$, $Y=Y_{\infty}$ and $Y=Y_p$. For $Y\in V(\X)$ with $Y=Y_2$ let $\Cox(Y)$ be the Coxeter polytope associated to the Coxeter group $D_Y$ endowed with its natural Euclidean metric as described in Section \ref{Coxeter}. Its set of vertices is $D_{Y}$. For $Y\in V(\X)$ with $Y = Y_{\infty}$ consider $\Pi_{v\in Y}[0,1] = [0,1]^{|Y|}\subset \R^{|Y|}$ with its standard cubical structure. Its set of vertices is $\calP(Y)$, where $0\in\R^{|Y|}$ corresponds to $\emptyset\in\calP(Y)$. For $v\in V_p$ let $\stern(v)$ be the $f(v)$-branched star where each edge of the star is identified with $[0,1]$. Its center is denoted $c_v$ and its tips are identified with the elements of $C_{f(v)}$ the finite cyclic group of order $f(v)$. For $Y\in V(\X)$ with $Y = Y_p$ let $\stern(Y)$ be the product of stars $\Pi_{v\in Y}\stern(v)$ endowed with the $\ell_2$ metric. So its vertex set is $\Pi_{v\in Y}(\{c_v\}\cup C_{f(v)})$. Note that $V(\stern(Y)) = \Pi_{v\in Y}(\{c_v\}\cup C_{f(v)})$ can be identified with  $\bigcup_{Z\subset Y} D_{Y}/D_Z$. Indeed we have $\{c_v\}\cup C_{f(v)} = D_{\{v\}}/D_{\{v\}} \cup D_{\{v\}}/D_{\emptyset}$. Let us denote a vertex $gD_Z\in D_{Y}/D_Z$ in $\stern(Y)$ with $gZ$.

\paragraph{The cell complex $\rC(Y)$} To every $Y\in V(\X)$ we associate a Euclidean cell complex $\rC(Y)$ as follows. Let $\rC(Y)$ be the product $\Cox(Y_2)\times[0,1]^{|Y_{\infty}|}\times\stern(Y_p)$ endowed with the $\ell_2$ metric. Each of its factors is piecewise Euclidean, so it is a piecewise Euclidean cell complex. In particular $\rC(Y)=\rC(Y_2)\times\rC(Y_{\infty})\times \rC(Y_p)$. The set of vertices of $\rC(Y)$ is $D_{Y_2}\times\calP(Y_{\infty})\times\Pi_{v\in Y_p}(\{c_v\}\cup C_{f(v)})$. The group $D^f_Y$ acts by isometries on $\rC(Y)$. Indeed $D^f_Y = D_{Y_2}\times\Pi_{v\in Y_p}C_{f(v)}$. So $D^f_Y$ acts through $D_{Y_2}$ on $\Cox(Y_2)$ and through $C_{f(v)}$ on $\stern(v)$. These actions are all isometries.

\begin{Bemerkung}[Links of vertices]\label{links of blocks} As we will need to understand links of vertices in subcomplexes we start by studying links of vertices in $\rC(Y)$.
	Consider a vertex $l=(w, \lambda, gZ)\in \rC(Y)$. Its link is the join $$\Lk(w,\Cox(Y_2))\star\Lk(\lambda, [0,1]^{|Y_{\infty}|})\star\Lk(gZ, \stern(Y_p)).$$ The link $\Lk(gZ, \stern(Y_p))$ is isometric to $ \Lk(Z, \stern(Y_p))$ which is a join  $$\star_{v\in Z}\Lk(c_v,\stern(v))\star_{v\in Y_p\setminus Z}\Lk(\neutre,\stern(v)).$$ Each component $\Lk(c_v,\stern(v))$ consists of $|C_{f(v)}|$ disjoint vertices and each component $\Lk(\neutre,\stern(v))$ consists of one vertex. We can decompose \linebreak $\Lk(\lambda, [0,1]^{|Y_{\infty}|}) = \left(\star_{v\in\lambda}\lambda\setminus\{v\}\right)\star\left(\star_{v\in Y_{\infty}\setminus\lambda} \lambda\cup\{v\}\right)$. So the link $\Lk(l, \rC(Y))$ is \begin{multline*} \Lk(w,\Cox(Y_2))\star\left(\left(\star_{v\in\lambda}\lambda\setminus\{v\}\right)\star\left(\star_{v\in Y_{\infty}\setminus\lambda} \lambda\cup\{v\}\right)\right) \\ \star(\star_{v\in Z}\Lk(c_v,\stern(v)))\star(\star_{v\in Y_p\setminus Z}\Lk(e,\stern(v))),\end{multline*} where the length of an edge between two components of any join is $\pi/2$ and $\Lk(w,\Cox(Y_2))$ is the flag complex over $\Gamma_{Y_2}$ with edge length $d(u,v) = \pi-\pi/m(u,v)$ by Section \ref{Coxeter}.
\end{Bemerkung}

\begin{Beispiel}
	Let $m, p\in \N_{\geq 2}$. We go back to the example of the Dyer graph $\Gamma_{m,p}$ with associated Dyer group $D_{m,p}$ and scwol $\X_{m,p}$ given in Figure \ref{fig1}, Example \ref{nontrivial example} and Figure \ref{exscwol}. Figure \ref{blocks} shows the cell complexes $\rC(\{a,b\})$, $\rC(\{b,c\})$, $\rC(\{c,d\})$ in the case $m=4$ and $p=3$.
\end{Beispiel}

\begin{figure}[h]
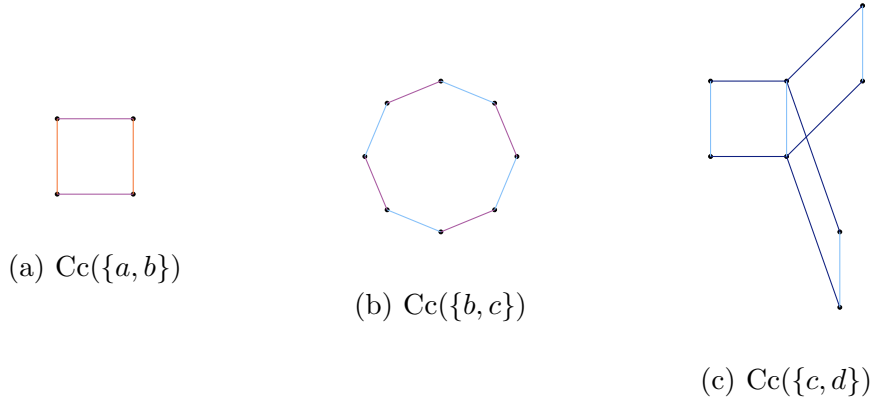

	\begin{subfigure}{.3\textwidth}
		\centering
		\include{abCell}
		\caption{$\rC(\{a,b\})$}
		\label{abCell}
	\end{subfigure}
	\begin{subfigure}{.3\textwidth}
		\centering
		\include{bcCell}
		\caption{$\rC(\{b,c\})$}
		\label{bcCell}
	\end{subfigure}
	\begin{subfigure}{.3\textwidth}
		\centering
		\include{cdCell}
		\caption{$\rC(\{c,d\})$}
		\label{cdCell}
	\end{subfigure}
	\caption{The cell complexes associated to some vertices of $\X_{m,p}$. 
	}
	\label{blocks}
\end{figure}

\paragraph{The cell complex $\Sigma(gY)$} Let $gY\in V(\C)$. We now describe the subcomplexes of $\Sigma$ associated to vertices of $\C$.
We start by identifying the vertex set of $\rC(Y)$ with a subset of $V(\St_{in}(Y, \C))$ and more generally with a subset of $V(\St_{in}(gY,\C))$. Let $V_p(gY)$ be the following subset of $V(\St_{in}(gY, \C))$: $$V_p(gY) = \{kX\in V(\St_{in}(gY,\C))\mid X\subset V_p\}.$$

\begin{Lemma}\label{identification} The map $j: V(\rC(Y)) \rightarrow V_p(\neutre Y)$ given by $j(w, \lambda, hZ) = w\phi(\lambda)hZ$ is bijective. Moreover it induces a bijective map $j_g: V(\rC(Y)) \rightarrow V_p(gY)$ with $j_g(w, \lambda, hZ) =  g\cdot j(w, \lambda, hZ)= gw\phi(\lambda)hZ$. 
\end{Lemma}

\begin{proof}
	For every vertex $kZ\in V_p(\neutre Y)$, we have $Z\subset Y_p$ and there is a unique decomposition $k = k_2k_{\infty}\Pi_{v\in Y_p\setminus Z}k_v$ with $k_2\in D_{Y_2}$, $k_v\in C_{f(v)}$ and $k_{\infty}\in D_{Y_{\infty}}$. Since $kZ\in V_p(\neutre Y)$ we have that $k_{\infty} = \phi(\lambda)$ for a unique $\lambda\subset (Y\setminus Z)_{\infty}$. So the map $j$ is bijective. Finally $kZ\in V_p(gY)$ if and only if $g^{-1}kZ \in V_p(\neutre Y)$. So the map $j_g$ is bijective. 
\end{proof}

For $gY\in \C$ let $\Sigma(gY)$ be the piecewise Euclidean cell complex isometric to $\rC(Y)$ with vertex set $V(\Sigma(gY))=V_p(gY)$ induced by the map $j_g$. Note that if $hY = gY$ the map $j_g\circ j_h^{-1}$ induces an isometry of $\Sigma(gY)$ so $\Sigma(gY)$ is well defined. More precisely $U\subset V_p(gY)$ is the vertex set of a cell of $\Sigma(gY)$ if and only if $j_g^{-1}(U)$ is the vertex set of a cell in $\rC(Y)$.

We now discuss identifications of subcomplexes. Indeed for $gY\in V(\C)$ and $hZ\in\St_{in}(gY,\C)$, we have $V_p(hZ)\subset V_p(gY)$. The following lemma shows that this inclusion induces an isometric embedding of the cell complex $\Sigma(hZ)$ into $\Sigma(gY)$.

\begin{Lemma}\label{inclusion}
	Let $gY\in V(\C)$ and $hZ\in\St_{in}(gY,\C)$. The map $\iota: \Sigma(hZ)\rightarrow\Sigma(gY)$ induced by $\iota(v) = v$ for every vertex $v\in \Sigma(hZ)$ is an isometric embedding preserving the cell structure. In particular we can identify $\Sigma(hZ)$ with $\iota(\Sigma(hZ))$. 
\end{Lemma}

\begin{proof}
Since $hZ\in\St_{in}(gY,\C)$ if and only if $g^{-1}hZ\in\St_{in}(Y,\C)$ it suffices to consider the case $g=\neutre$. For $hZ\in\St_{in}(\neutre Y,\C)$ we can write $h = h_2h_{\infty}h_p$ with $h_2\in D_{Y_2}$, $h_{\infty} = \phi(\kappa)$ for a unique $\kappa\subset (Y\setminus Z)_{\infty}$ and $h_p\in D_{Y_p\setminus Z_p}$. The map $\iota_h:\rC(Z)\rightarrow\rC(Y)$ induced by $\iota(w,\lambda,mM) = (h_2w,\lambda\cup\kappa,h_pmM)$ is an isometric embedding preserving the cell structure. Indeed this holds for the restrictions to $\rC(Z_2)$ (where $\iota(w) = h_2w$ which corresponds to identifying $\Cox(Z_2)$ with $h_2\cdot\Cox_{Z_2}(Y_2)$), $\rC(Z_{\infty})$ (where $\iota(\lambda) =  \lambda\cup\kappa$ which identifies $\Pi_{v\in Z_{\infty}}[0,1]$ with $\Pi_{v\in Z_{\infty}}[0,1]\times\Pi_{v\in\kappa}\{1\}\times\Pi_{v\in Y_{\infty}\setminus(\kappa\cup Z_{\infty})}\{0\}\subset \Pi_{v\in Y_{\infty}}[0,1]$ ) and $\rC(Z_p)$ (where $\iota(mM) =  h_pmM$ which identifies $\stern(Z_p)$ with $\stern(Z_p)\times\{h_p\}\subset\stern(Y_p)$). Then the map $\iota:  \Sigma(hZ)\rightarrow\Sigma(\neutre Y)$ is induced by $j_h^{-1}\circ\iota_h\circ j_{\neutre}$ so $\iota$ is an isometric embedding preserving the cell structure.
\end{proof}

The next lemma discusses how $\St_{in}(gY, \C)$ can be interpreted as a simplicial subdivision of $\Sigma(gY)$.

\begin{Lemma}\label{simplicial subdivision}
	Let $gY\in V(\C)$. The scwol $\St_{in}(gY, \C)$ can be realized as the scwol of a simplicial subdivision of $\Sigma(gY)$. Moreover for $hZ\in\St_{in}(gY,\C)$, the isometric embedding $\iota: \Sigma(hZ)\rightarrow\Sigma(gY)$ given in Lemma \ref{inclusion} preserves the simplicial subdivision.
\end{Lemma}

\begin{proof}
	Let $Y\in V(\X)$ and $g\in D$. Then using Lemma \ref{Coxeter description}, $\St_{in}(Y_2,\C)$ is isomorphic as a scwol to the barycentric subdivision of $\Cox(Y_2)$. For $v\in Y_{\infty}$, $\St_{in}(\{v\},\C)$ is isomorphic as a scwol to the barycentric subdivision of $[0,1]$. For $v\in Y_{p}$, $\St_{in}(\{v\},\C)$ is an $f(v)$-branched star so as a simplicial complex it is isomorphic to $\stern(v)$. Using the product decompositions $$\rC(Y) = \rC(Y_2)\times\Pi_{v\in Y_{\infty}}\rC(\{v\})\times\Pi_{v\in Y_p}\rC(\{v\})$$ and $$\St_{in}(Y,\C) = \St_{in}(Y_2,\C)\times\Pi_{v\in Y_{\infty}}\St_{in}(\{v\},\C)\times\Pi_{v\in Y_p}\St_{in}(\{v\},\C),$$ we have that $\St_{in}(Y,\C)$ describes a simplicial subdivision of $\rC(Y)$. Hence \linebreak $\St_{in}(gY,\C)$ is the scwol of this simplicial subdivision of $\Sigma(gY)$. 
	 This induces a metric on the geometric realization $|\St_{in}(gY, \C)|$ such that it is isometric to $\Sigma(gY)$. For the second statement note that $\iota$ is an isometric embedding, preserves the cell structure, and decomposes as a product. Additionally the canonical inclusion $\St_{in}(hZ, \C)\rightarrow\St_{in}(gY, \C)$ also decomposes as a product. The simplicial subdivision is preserved for each factor of the product decomposition. So the simplicial subdivision is preserved by $\iota$.
\end{proof}

\begin{Beispiel}
	Let $m, p\in \N_{\geq 2}$. We go back to the example of the Dyer graph $\Gamma_{m,p}$ with associated Dyer group $D_{m,p}$ and scwol $\X_{m,p}$ given in Figure \ref{fig1}, Example \ref{nontrivial example} and Figure \ref{exscwol}. Figure \ref{subcomplex} shows the subcomplexes $\Sigma(\neutre\{a,b\})$, $\Sigma(\neutre\{b,c\})$, $\Sigma(\neutre\{c,d\})$ and their simplicial subdivision in the case $m=4$ and $p=3$.
\end{Beispiel}

\begin{figure}[h]
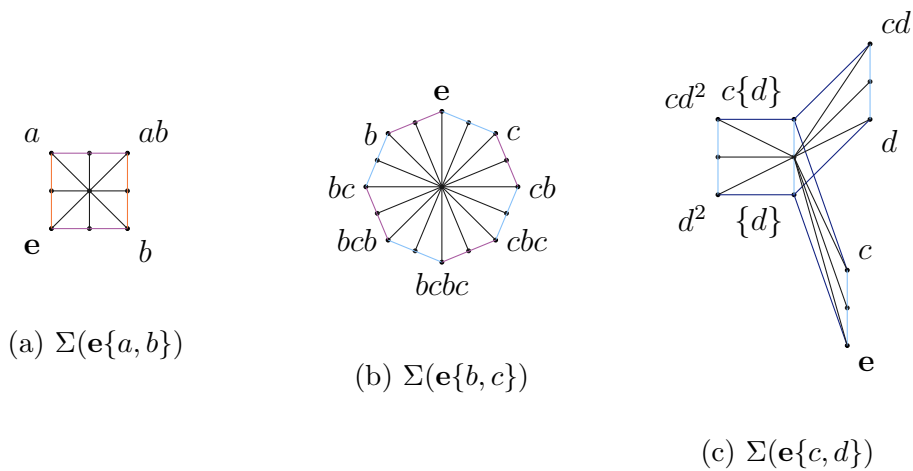

	\begin{subfigure}{.3\textwidth}
		\centering
		\include{abSigma}
		\caption{$\Sigma(\neutre\{a,b\})$}
		\label{abSigma}
	\end{subfigure}
	\begin{subfigure}{.3\textwidth}
		\centering
		\include{bcSigma}
		\caption{$\Sigma(\neutre\{b,c\})$}
		\label{bcSigma}
	\end{subfigure}
	\begin{subfigure}{.3\textwidth}
		\centering
		\include{cdSigma}
		\caption{$\Sigma(\neutre\{c,d\})$}
		\label{cdSigma}
	\end{subfigure}
\caption{The subcomplexes associated to some vertices of the development of $\X_{m,p}$ and their simplicial subdivision. 
}
\label{subcomplex}
\end{figure}

\paragraph{The cell complex $\Sigma$} We now have the tools needed to build the cell complex $\Sigma$. Consider $$\Sigma = \bigcup_{gY\in V(\C)} \Sigma(gY),$$ where we identify $\Sigma(hZ)$ with $\iota(\Sigma(hZ))\subset\Sigma(gY)$ whenever $hZ\in \St_{in}(gY,\C)$. So by Lemma \ref{inclusion}, $\Sigma$ has a well-defined piecewise Euclidean metric. We endow $\Sigma$ with the associated length metric. The set of vertices of $\Sigma$ is $$V_p(\C) = \{gY\in V(\C)\mid Y\subset V_p\}.$$ The action of $D$ on $V_p(\C)$ induces an action by isometries of $D$ on $\Sigma$, in particular for $d\in D$ we have $d\cdot\Sigma(gY) = \Sigma(dgY)$.

\begin{Lemma}\label{simply connected}
	The scwol $\C$ describes a simplicial subdivision of $\Sigma$. In particular this implies that $\Sigma$ is a simply connected metric space.
\end{Lemma}

\begin{proof}
	Since $\C = \bigcup_{gY\in V(\C)}\St_{in}(gY,\C)$ and by Lemma \ref{simplicial subdivision} every $\St_{in}(gY,\C)$ is a simplicial subdivision of $\Sigma(gY)$ preserved by $\iota$, the complex $\C$ is a simplicial subdivion of $\Sigma$. This induces a metric on $\C$ with respect to which the geometric realization $|\C|$ is isometric to $\Sigma$. So $\Sigma$ is a well-defined simply connected metric space.
\end{proof}

We are finally in a position to show that $\Sigma$ is CAT(0). Since $\Sigma$ is simply connected, we only need to understand its local structure, so we are back to studying links of vertices. In order to have a precise description of the links of vertices we introduce an edge labeling of $\Sigma$ by $V(\Gamma)$.

\paragraph{Edge labeling} Let $hY\in \C$. We start by labeling the edges of $\Sigma(hY)$ by elements of $Y$. To define this edge labeling we identify when two vertices of $\Sigma(h Y)$ are adjacent and then give the corresponding label. Let $kX, lZ\in V_p(h Y)$, i.e. they are vertices of $\Sigma(h Y)$. Then $kX$ and $lZ$ are adjacent in $\Sigma(h Y)$ if and only if their pre-images $j_h^{-1}(kX), j_h^{-1}(lZ)\in V(\rC(Y))$ are adjacent in $\rC(Y)$. This leads to the following characterization and labeling of edges by $Y\subset V(\Gamma)$. The vertices $kX, lZ\in V_p(h Y)$ are adjacent in $\Sigma(h Y)$ if and only if one of the following holds \begin{enumerate}
	\item $X=Z$ and $k^{-1}l = x_v$ for some $v\in Y_2$. In this case we label the edge by $v\in Y_2\subset V(\Gamma)$.
	\item $X=Z$ and $k^{-1}l = x_v^{\pm 1}$ for some $v\in Y_{\infty}$. In this case we label the edge by $v\in Y_{\infty}\subset  V(\Gamma)$.
	\item $X\subset Z$ and $Z\setminus X = \{x_v\}$ for some $v\in Y_p$ and $k^{-1}l\in \langle x_v\rangle$. In this case we label the edge by $v\in Y_{p}\subset  V(\Gamma)$.
	\item $Z\subset X$ and $X\setminus Z = \{x_v\}$ for some $v\in Y_p$ and $k^{-1}l\in \langle x_v\rangle$. In this case we label the edge by $v\in Y_{p}\subset  V(\Gamma)$.
\end{enumerate}

Note that for $h'Y'\in\St_{in}(gY,\C)$, the labeling of an edge in $\Sigma(h'Y')$ is invariant under the inclusion $\iota: \Sigma(h'Y')\rightarrow\Sigma(gY)$. Moreover the labeling of edges in $\Sigma(\neutre Y)$ is invariant under the action of $D^f_Y$. So this defines a labeling by $V(\Gamma)$ of the edges of $\Sigma$. Note that this edge labeling is invariant under the action of $D$.

\begin{Bemerkung}[Links of vertices]\label{links of vertices} As our goal is to apply Moussong's Lemma to $\Sigma$ we need to understand links of vertices in $\Sigma$. We start with links of vertices in $\Sigma(gY)$. This is crucial to prove later that $\Sigma$ is CAT(0). 
Let $hY\in V(\C)$, $kX\in V_p(hY)$. The edge labeling on $\Sigma$ and $\Sigma(gY)$ induce a vertex labeling $l:V(\Lk(kX,\Sigma))\rightarrow V$, which restricts to $l:V(\Lk(kX,\Sigma(hY)))\rightarrow Y$. Using the map $j_h$ in Lemma \ref{identification} and Remark \ref{links of blocks}, we see that $\Lk(kX,\Sigma(hY))$ is the flag complex over the join $\Gamma_{Y_2}\star\Gamma_{Y_{\infty}}\star\Gamma_{Y_p\setminus X}\star(\star_{v\in X}\{v^i\mid 1\leq i\leq f(v)\})$. The vertex labeling is given by $l(v) = v$ for every $v\in Y_2\cup Y_{\infty}\cup Y_p\setminus X$ and $l(v^i)= v$ for every $v^i\in\{v^i\mid v\in X, 1\leq i\leq f(v)\}$. 
By Remark \ref{links of blocks}, the edge length in $\Lk(kX,\Sigma(hY))$ is given by $d(v,w)=\pi-\pi/m(l(v),l(w))$. In particular Lemma \ref{posdef} implies that $\Lk(kX,\Sigma(hY))$ is a metric flag complex. It follows that $v, w$ are adjacent vertices in $\Lk(kX,\Sigma(hY))$ if and only if $l(v)\neq l(w)$. As this holds for every $gY\in V(\C)$, it implies that if $v, w$ are adjacent vertices in $\Lk(kX,\Sigma)$, we have $l(v)\neq l(w)$. So for pairwise adjacent vertices $v_1,\dots,v_n\in\Lk(kX,\Sigma)$, we have $l(v_i)\neq l(v_j)$ for every $i\neq j$. So in a slight abuse of notation we will write $v_i=l(v_i)\in V$ when considering pairwise adjacent vertices $v_1,\dots,v_n\in\Lk(kX,\Sigma)$.
\end{Bemerkung}

Let us now give some more details on the link of vertices in $\Sigma$.
\begin{Lemma}\label{lem:link} Let $Y\in V_p(\C)$. Let $v_1,\dots,v_k\in\Lk(Y,\Sigma)$ be pairwise adjacent vertices. There exist $Z\in V(\X)$ and $g\in D$ such that $Y\in\Sigma(gZ)$ and \linebreak $v_1,\dots,v_k\in\Lk(Y,\Sigma(gZ))$ if and only if $Y\cup\{v_1,\dots,v_k\}\in V(\X)$.
\end{Lemma}
\begin{proof}
	As $v_1,\dots,v_k\in\Lk(Y,\Sigma)$ are pairwise adjacent, $l(v_i)\neq l(v_j)$ so we write $v_i=l(v_i)\in V$ as mentioned in Remark \ref{links of vertices}. Assume there exists $Z\in V(\X)$ and $g\in D$ such that $Y\in\Sigma(gZ)$ and $v_1,\dots,v_k\in\Lk(Y,\Sigma(gZ))$. Then $Y\in V_p(gZ)$ so $Y\subset Z$ and $\{v_1,\dots,v_k\}=\{l(v_1),\dots,l(v_k)\}\subset Z$. Hence $Y\cup\{v_1,\dots,v_k\}\subset Z$ which implies $Y\cup\{v_1,\dots,v_k\}\in V(\X)$.
	
	Now assume  $Y\cup\{v_1,\dots,v_k\}\in V(\X)$. Each vertex $v\in\Lk(Y,\Sigma)$ corresponds to an edge in $\Sigma$ between $\neutre Y$ and some vertex $h_vZ_v\in\Sigma.$
	\begin{enumerate}[label=(\roman*)]
		\item If $v\in V_2$, the vertex $v\in \Lk(Y,\Sigma)$ corresponds to an edge between $Y$ and $x_{v}Y$. In this case fix $g_v=\neutre$. 
		\item If $v\in V_{\infty}$, the vertex $v\in \Lk(Y,\Sigma)$ corresponds to an edge between $Y$ and $\phi(v)Y$ or between $Y$ and $\phi(v)^{-1}Y$. In the first case fix $g_v = \neutre$. 
		 In the second case fix $g_v = \phi(v)^{-1}=x_v^{-1}$. Note that only one of these cases can occur as $v_1,\dots,v_k$ are pairwise adjacent.
		\item For $v\in V_p\setminus Y$, the vertex $v\in\Lk(Y,\Sigma)$ corresponds to an edge between $Y$ and $Y\cup\{v\}$. In this case we fix $g_v = \neutre$.
		\item For $v \in Y$, the vertex $v\in \Lk(Y,\Sigma)$ corresponds to an edge between $Y$ and $x_{v}^t(Y\setminus\{v\})$, for some $1\leq t\leq f(v)$. In this case fix $g_v = \neutre$.
	\end{enumerate}
	We claim that for $Z=Y\cup\{v_1,\dots,v_k\} $ and $g = \Pi_{i=1}^kg_{v_i}$ we have $Y\in\Sigma(gZ)$ and $v_1,\dots,v_k\in\Lk(Y,\Sigma(gZ))$. Since $Y\cup\{v_1,\dots,v_k\}\in V(\X)$, we have $g_vg_w=g_wg_v$ for all $v,w\in\{v_1,\dots,v_k\}$. In fact $g = \phi(\omega)^{-1}$ for $\omega = \{v\in Z\mid g_v = x_v^{-1}\}\subset (Z\setminus Y)_{\infty}$. Hence $Y\in V_p(gZ)$. Let $v\in\{v_1,\dots,v_k\}$. Now we need to show that $h_vZ_v\in V_p(gZ)$. We use the case by case analysis above to fix the following notation.
	\begin{enumerate}[label=(\roman*)]
		\item If $v\in V_2$, we have $h_v Z_v = x_{v}Y$ and we set $\lambda_v = \omega \subset (Z\setminus Z_v)_{\infty}$. 
		\item If $v\in V_{\infty}$ and $h_v Z_v = \phi(v)Y$, let $\lambda_v = \omega\cup \{v\} \subset (Z\setminus Z_v)_{\infty}$. If $v\in V_{\infty}$ and $h_vZ_v = \phi(v)^{-1}Y$ , let $\lambda_v = \omega\setminus\{v\}\subset (Z\setminus Z_v)_{\infty}$. 
		\item If $v\in V_p\setminus Y$, we have $h_vZ_v = Y\cup\{v\}$ and we set $\lambda_v = \omega \subset (Z\setminus Z_v)_{\infty}$. 
		\item If $v\in Y$, we have $h_vZ_v = x_{v}^t(Y\setminus\{v\})$ for some $1\leq t\leq f(v)$ and we set $\lambda_v = \omega \subset (Z\setminus Z_v)_{\infty}$. 
	\end{enumerate}
	 As $Z =Y\cup\{v_1,\dots,v_k\} \in V(\X)$, we have $gZ = h_v\phi(\lambda_v)^{-1}Z$ and $Z_v\subset Z$, so $h_vZ_v\in V(\St_{in}(gZ,\C))$. As additionally $Z_v\subset Z\cap V_p$, we have $h_vZ_v\in V_p(gZ)$. 
\end{proof}

We now have the necessary tools to show the following statement.

\begin{Theorem}\label{thm:cat}
	The cell complex $\Sigma$ is CAT(0).
\end{Theorem}

\begin{proof}
	 By \cite{BriHae} Theorem II.5.4, $\Sigma$ is CAT(0) if and only if it is simply connected and the link of every vertex is CAT(1). By Lemma \ref{simply connected}, the cell complex $\Sigma$ is simply connected. Let us now prove that the link of every vertex is CAT(1) by using Moussong's Lemma \ref{Moussong}. Consider $gY\in V(\Sigma)$ with $g=\neutre$. So $gY = \neutre Y=Y$.
	  \begin{Claim}
	 	Every edge in the link $\Lk(Y,\Sigma)$ of $Y$ in $\Sigma$ has length $\geq \pi/2$.
	 \end{Claim} 
 
 \begin{proof}
 	Since the vertex $Y\in\Sigma$ is contained in $\Sigma(gZ)$ if and only if $gZ\in \St_{out}(Y,\C)$ we can describe
 	$\Lk(Y,\Sigma)$ as the union $\bigcup_{gZ\in \St_{out}(Y,\C)}\Lk(Y,\Sigma(gZ))$, where \linebreak $\Lk(Y,\Sigma(gZ))$ is the link of $Y$ in the subcomplex $\Sigma(gZ)$. By Remark \ref{links of vertices}, the length of edges in $\Lk(Y,\Sigma(gZ))$ satisfy $d(u,v) = \pi-\pi/m(l(u),l(v)) \geq \pi/2$ as $m(l(u),l(v)) \geq 2$. So each edge in $\Lk(Y,\Sigma)$ has length $\geq \pi/2$.
 \end{proof}

	\begin{Claim}
	The link $\Lk(Y,\Sigma)$ of $Y$ in $\Sigma$ is metrically flag.
\end{Claim}

\begin{proof}
	Consider a set of pairwise adjacent vertices $v_1, \dots, v_k\in \Lk(Y,\Sigma)$. As mentioned in Remark \ref{links of vertices}, so $l(v_i)\neq l(v_j)$, so we write $v_i=l(v_i)\in V$. In particular $Y\cup\{v_1,\dots,v_k\}$ spans a complete subgraph of $\Gamma$. So $v_1,\dots,v_k$ span a simplex in $\Lk(Y,\Sigma)$ if and only $v_1, \dots, v_k$ span a simplex in $\Lk(Y,\Sigma(gZ))$ for some $gZ\in V(\C)$. By Remark \ref{links of vertices} $\Lk(Y,\Sigma(gZ))$ is flag. So $v_1, \dots, v_k\in \Lk(Y,\Sigma)$ span a simplex in $\Lk(Y,\Sigma)$ if and only if there exists some $gZ\in V(\C)$ with $Y\in \Sigma(gZ)$ and $v_1,\dots,v_k\in\Lk(Y,\Sigma(gZ))$. By Lemma \ref{lem:link}, this is the case if and only if $Y'=Y\cup\{v_1,\dots,v_k\}\in V(\X)$. By Fact \ref{posdef}, $Y'\in V(\X)$ if and only if the matrix $(\cos(\pi-\pi/m(u,v)))_{u,v\in Y'}$ is positive definite. As $\pi-\pi/m(u,v) = \pi/2$ for all $u\in Y'\setminus V_2$, $v\in Y'\setminus\{u\}$ and $\pi-\pi/m(u,u) = 0$ for all $u\in Y'$, the matrix $(\cos(\pi-\pi/m(u,v)))_{u,v\in Y'}$ is positive definite if and only if its restriction $(\cos(\pi-\pi/m(u,v)))_{u,v\in Y'\cap V_2}$ is positive definite. As $\pi-\pi/m(u,v) = d(u,v)$ for all $u,v\in Y'\cap V_2$, $(\cos(\pi-\pi/m(u,v)))_{u,v\in Y'\cap V_2}$ is positive definite if and only if $(\cos(d(u,v))_{u,v\in Y'\cap V_2}$ is positive definite. As ${u,v\in Y'\cap V_2} = \{v_1,\dots,v_k\}\cap V_2$, we have $(\cos(d(u,v))_{u,v\in Y'\cap V_2}$ is positive definite if and only if $(\cos(d(u,v))_{u,v\in \{v_1,\dots,v_k\}\cap V_2}$ is positive definite. Finally $d(u,u) = 0$ for all $u\in \{v_1,\dots,v_k\}$, and $d(u,v) = \pi/2$ for all $u\in \{v_1,\dots,v_k\}\setminus V_2$, $v\in \{v_1,\dots,v_k\}\setminus\{u\}$, so the matrix \linebreak $(\cos(d(u,v))_{u,v\in \{v_1,\dots,v_k\}\cap V_2}$ is positive definite if and only if $(\cos(d(u,v))_{u,v\in \{v_1,\dots,v_k\} }$ is positive definite. So we conclude that $v_1, \dots, v_k\in \Lk(Y,\Sigma)$ span a simplex if and only if $(\cos(d(u,v))_{u,v\in \{v_1,\dots,v_k\} }$ is positive definite. So $\Lk(Y,\Sigma)$ is metrically flag.
\end{proof}

So by Moussong's Lemma, $\Lk(Y,\Sigma)$ is CAT(1). Since $D$ acts by isometries on $\Sigma$, the link $\Lk(gY, \Sigma)$ is CAT(1) for every $g\in D$. We conclude that $\Sigma$ is CAT(0).

\end{proof}

\begin{Bemerkung}
	If $D$ is a spherical Dyer group the scwol $\C$ decomposes as a product $\C_2\times\C_p\times\C_{\infty}$. For $i\in\{2,p,\infty\}$ let $\Sigma_i$ be the cell complex associated to $D_i$. So $\Sigma_2 = \Cox(V_2)$, $\Sigma_{\infty} = \R^{|V_{\infty}|}$ and $\Sigma_p = \stern(V_p)$. Then $\Sigma = \Sigma_2\times\Sigma_{\infty}\times\Sigma_p$ where each factor is known to be CAT(0). So $\Sigma$ is CAT(0).
\end{Bemerkung}

\begin{Korollar}
	The Dyer group $D$ is CAT(0).
\end{Korollar}
\begin{proof}
	$D$ acts properly discontinuously and cocompactly by isometries on $\Sigma$.
\end{proof}

\begin{Bemerkung} If the Dyer group $D$ is a Coxeter group, $\Sigma$ is the Davis-Moussong complex described in Theorem \ref{Davis}. If the Dyer group $D$ is a right-angled Artin group, $\Sigma$ is the Salvetti complex described in Section \ref{sec:Salvetti}. The dimension of $\Sigma$ is $\dim(\Sigma)=\max\{|Y|\mid Y\in V(\X)\}$. Consider the Coxeter group $W$ from Theorem \ref{Dyer} and its associated Davis-Moussong complex $\Sigma(W)$. The dimension of $\Sigma(W)$ is $\dim(\Sigma(W))=\max\{|S| \mid S\subset V(\Lambda), W_S \text{ is finite } \}$. Looking at the construction of the graph $\Lambda$ we can see that $\dim(\Sigma(W))=\max\{|Y|+|V_p|+|V_{\infty}\setminus Y| \mid Y\in V(\X)\}$. So $\dim(\Sigma)\leq \dim(\Sigma(W))$.

\end{Bemerkung}

%% file: abCell.tex
\begin{tikzpicture}
\coordinate (neutre) at (0,0);
\coordinate (a) at (0,1);
\coordinate (b) at (1,0);
\coordinate (ab) at (1,1);

\draw[fill] (neutre) circle [radius=0.025];
\draw[fill] (a) circle [radius=0.025];
\draw[fill] (b) circle [radius=0.025];
\draw[fill] (ab) circle [radius=0.025];


\draw [violet](neutre)--(b);
\draw [orange] (b)--(ab);
\draw [violet] (ab)--(a);
\draw [orange](a)--(neutre);


\end{tikzpicture}

%% file: bcCell.tex
\begin{tikzpicture}

\coordinate (neutre) at (90:1);
\coordinate (b) at (135:1);
\coordinate (bc) at (180:1);
\coordinate (bcb) at (225:1);
\coordinate (bcbc) at (270:1);
\coordinate (c) at (45:1);
\coordinate (cb) at (0:1);
\coordinate (cbc) at (315:1);

\draw[fill] (neutre) circle [radius=0.025];
\draw[fill] (b) circle [radius=0.025];
\draw[fill] (bc) circle [radius=0.025];
\draw[fill] (bcb) circle [radius=0.025];
\draw[fill] (bcbc) circle [radius=0.025];
\draw[fill] (c) circle [radius=0.025];
\draw[fill] (cb) circle [radius=0.025];
\draw[fill] (cbc) circle [radius=0.025];


\draw [violet](neutre)--(b);
\draw [hellblau](b)--(bc);
\draw [violet](bc)--(bcb);
\draw[hellblau](bcb)--(bcbc);
\draw [violet](bcbc)--(cbc);
\draw[hellblau] (cbc)--(cb);
\draw[violet] (cb)--(c);
\draw [hellblau](c)--(neutre);



\end{tikzpicture}

%% file: cdCell.tex
\begin{tikzpicture}
\coordinate (ed) at (0,0);
\coordinate (d2) at (-1,0);
\coordinate (neutre) at (0.7,-2);
\coordinate (d) at (1,1);
\coordinate (ced) at (0,1);
\coordinate (cd) at (1,2);
\coordinate (cd2) at (-1,1);
\coordinate (c) at (0.7,-1);

\draw[fill] (neutre) circle [radius=0.025];
\draw[fill] (d2) circle [radius=0.025];
\draw[fill] (ed) circle [radius=0.025];
\draw[fill] (d) circle [radius=0.025];
\draw[fill] (ced) circle [radius=0.025];
\draw[fill] (cd) circle [radius=0.025];
\draw[fill] (cd2) circle [radius=0.025];
\draw[fill] (c) circle [radius=0.025];


\draw[dunkelblau] (ced)--(cd);
\draw [dunkelblau](ced)--(cd2);
\draw [dunkelblau](ced)--(c);
\draw [dunkelblau](ed)--(neutre);
\draw [dunkelblau](ed)--(d);
\draw [dunkelblau](ed)--(d2);
\draw [hellblau](d)--(cd);
\draw [hellblau](neutre)--(c);
\draw[hellblau] (d2)--(cd2);
\draw[hellblau] (ed)--(ced);



\end{tikzpicture}

%% file: abSigma.tex
\begin{tikzpicture}
\coordinate (neutre) at (0,0);
\coordinate (a) at (0,1);
\coordinate (b) at (1,0);
\coordinate (ab) at (1,1);

\draw[fill] (neutre) circle [radius=0.025] node[below left]{$\neutre$};
\draw[fill] (a) circle [radius=0.025]node[above left]{$a$};
\draw[fill] (b) circle [radius=0.025]node[below right]{$b$};
\draw[fill] (ab) circle [radius=0.025]node[above right]{$ab$};

\draw[fill, color=black!90!white] (0, 0.5) circle [radius=0.025];
\draw[fill, color=black!90!white] (0.5,0) circle [radius=0.025];
\draw[fill, color=black!90!white] (1, 0.5) circle [radius=0.025];
\draw[fill, color=black!90!white] (0.5,1) circle [radius=0.025];
\draw[fill, color=black!90!white] (0.5,0.5) circle [radius=0.025];

\draw [violet](neutre)--(b);
\draw [orange] (b)--(ab);
\draw [violet] (ab)--(a);
\draw [orange](a)--(neutre);

\draw[color=black!90!white] (0.5,0.5)--(neutre);
\draw[color=black!90!white] (0.5,0.5)--(0.5,0);
\draw[color=black!90!white] (0.5,0.5)--(b);
\draw[color=black!90!white] (0.5,0.5)--(1,0.5);
\draw[color=black!90!white] (0.5,0.5)--(ab);
\draw[color=black!90!white] (0.5,0.5)--(0.5,1);
\draw[color=black!90!white] (0.5,0.5)--(a);
\draw[color=black!90!white] (0.5,0.5)--(0,0.5);

\end{tikzpicture}

%% file: bcSigma.tex
\begin{tikzpicture}

\coordinate (neutre) at (90:1);
\coordinate (b) at (135:1);
\coordinate (bc) at (180:1);
\coordinate (bcb) at (225:1);
\coordinate (bcbc) at (270:1);
\coordinate (c) at (45:1);
\coordinate (cb) at (0:1);
\coordinate (cbc) at (315:1);

\draw[fill] (neutre) circle [radius=0.025] node[above]{$\neutre$};
\draw[fill] (b) circle [radius=0.025]node[left]{$b$};
\draw[fill] (bc) circle [radius=0.025]node[left]{$bc$};
\draw[fill] (bcb) circle [radius=0.025]node[left]{$bcb$};
\draw[fill] (bcbc) circle [radius=0.025] node[below]{$bcbc$};
\draw[fill] (c) circle [radius=0.025]node[right]{$c$};
\draw[fill] (cb) circle [radius=0.025]node[right]{$cb$};
\draw[fill] (cbc) circle [radius=0.025]node[right]{$cbc$};

\draw[fill,  color=black!90!white] (112.5:0.92) circle [radius=0.025];
\draw[fill,  color=black!90!white] (157.5:0.92) circle [radius=0.025];
\draw[fill,  color=black!90!white] (202.5:0.92) circle [radius=0.025];
\draw[fill,  color=black!90!white] (247.5:0.92) circle [radius=0.025];
\draw[fill,  color=black!90!white] (292.5:0.92) circle [radius=0.025];
\draw[fill,  color=black!90!white] (337.5:0.92) circle [radius=0.025];
\draw[fill,  color=black!90!white] (67.5:0.92) circle [radius=0.025];
\draw[fill,  color=black!90!white] (22.5:0.92) circle [radius=0.025];
\draw[fill,  color=black!90!white] (0,0) circle [radius=0.025];

\draw [violet](neutre)--(b);
\draw [hellblau](b)--(bc);
\draw [violet](bc)--(bcb);
\draw[hellblau](bcb)--(bcbc);
\draw [violet](bcbc)--(cbc);
\draw[hellblau] (cbc)--(cb);
\draw[violet] (cb)--(c);
\draw [hellblau](c)--(neutre);

\draw[color=black!90!white] (0,0)--(neutre);
\draw[color=black!90!white] (0,0)--(b);
\draw[color=black!90!white] (0,0)--(bc);
\draw[color=black!90!white] (0,0)--(bcb);
\draw[color=black!90!white] (0,0)--(bcbc);
\draw[color=black!90!white] (0,0)--(c);
\draw[color=black!90!white] (0,0)--(cb);
\draw[color=black!90!white] (0,0)--(cbc);

\draw[color=black!90!white] (0,0)--(112.5:0.92);
\draw[color=black!90!white] (0,0)--(157.5:0.92);
\draw[color=black!90!white] (0,0)--(202.5:0.92);
\draw[color=black!90!white] (0,0)--(247.5:0.92);
\draw[color=black!90!white] (0,0)--(292.5:0.92);
\draw[color=black!90!white] (0,0)--(337.5:0.92);
\draw[color=black!90!white] (0,0)--(67.5:0.92);
\draw[color=black!90!white] (0,0)--(22.5:0.92) ;

\end{tikzpicture}

%% file: cdSigma.tex
\begin{tikzpicture}
\coordinate (ed) at (0,0);
\coordinate (d2) at (-1,0);
\coordinate (neutre) at (0.7,-2);
\coordinate (d) at (1,1);
\coordinate (ced) at (0,1);
\coordinate (cd) at (1,2);
\coordinate (cd2) at (-1,1);
\coordinate (c) at (0.7,-1);

\draw[fill] (neutre) circle [radius=0.025] node[below right]{$\neutre$};
\draw[fill] (d2) circle [radius=0.025] node[below left]{$d^2$};
\draw[fill] (ed) circle [radius=0.025] node[below left]{$\{d\}$};
\draw[fill] (d) circle [radius=0.025] node[below right]{$d$};
\draw[fill] (ced) circle [radius=0.025] node[above left]{$c\{d\}$};
\draw[fill] (cd) circle [radius=0.025] node[above right]{$cd$};
\draw[fill] (cd2) circle [radius=0.025] node[above left]{$cd^2$};
\draw[fill] (c) circle [radius=0.025] node[above right]{$c$};

\draw[fill,  color=black!90!white] (1,1.5) circle [radius=0.025];
\draw[fill,  color=black!90!white] (0,0.5) circle [radius=0.025];
\draw[fill,  color=black!90!white] (-1,0.5) circle [radius=0.025];
\draw[fill,  color=black!90!white] (0.7,-1.5) circle [radius=0.025];

\draw[dunkelblau] (ced)--(cd);
\draw [dunkelblau](ced)--(cd2);
\draw [dunkelblau](ced)--(c);
\draw [dunkelblau](ed)--(neutre);
\draw [dunkelblau](ed)--(d);
\draw [dunkelblau](ed)--(d2);
\draw [hellblau](d)--(cd);
\draw [hellblau](neutre)--(c);
\draw[hellblau] (d2)--(cd2);
\draw[hellblau] (ed)--(ced);

\draw[color=black!90!white] (0,0.5)--(cd2);
\draw[color=black!90!white] (0,0.5)--(d2);
\draw[color=black!90!white] (0,0.5)--(-1,0.5);

\draw[color=black!90!white] (0,0.5)--(cd);
\draw[color=black!90!white] (0,0.5)--(d);
\draw[color=black!90!white] (0,0.5)--(1,1.5);

\draw[color=black!90!white] (0,0.5)--(c);
\draw[color=black!90!white] (0,0.5)--(neutre);
\draw[color=black!90!white] (0,0.5)--(0.7,-1.5);
\end{tikzpicture}